\documentclass[twoside,11pt]{article}

\usepackage{ifthen}
\newboolean{ElectronicVersion}
\setboolean{ElectronicVersion}{true} 

\usepackage{amsmath,amstext,amsthm} 
\usepackage{subfig}
\usepackage{booktabs}
\usepackage{epstopdf}
\usepackage{microtype}
\usepackage{graphicx}

\usepackage[sort,round]{natbib}
\usepackage{hyperref}
\usepackage[dvipsnames]{xcolor}
\hypersetup{colorlinks=true, allcolors=NavyBlue}

\usepackage[top=1in, bottom=1in, left=1in, right=1in]{geometry} 

\usepackage{algorithm}
\usepackage{algorithmicx}
\usepackage{algpseudocode}
\makeatletter
\newcommand{\algmargin}{\the\ALG@thistlm}
\makeatother
\newlength{\whilewidth}
\algdef{SE}[parWHILE]{parWhile}{EndparWhile}[1]
  {\parbox[t]{\dimexpr\linewidth-\algmargin}{%
     \hangindent\whilewidth\strut\algorithmicwhile\ #1\ \algorithmicdo\strut}}{\algorithmicend\ \algorithmicwhile}%
\algnewcommand{\parState}[1]{\State%
  \parbox[t]{\dimexpr\linewidth-\algmargin}{\strut #1\strut}}

\newtheorem{theorem}{Theorem}[section]
\newtheorem{proposition}{Proposition}[section]

\newtheorem{lemma}{Lemma}[section]

\renewcommand{\S}{\mathbf{S}}

\newcommand{\R}{\mathbf{R}}

\newcommand{\Z}{\bs{Z}}
\newcommand{\E}{\mathbf{E}}
\newcommand{\St}{\bs{\tilde S}}
\DeclareMathOperator*{\argmin}{\arg\min}
\DeclareMathOperator{\pr}{Pr}

\newcommand{\Diag}{\mathrm{Diag}}

\DeclareMathOperator{\rank}{{rank}}

\DeclareMathOperator{\tr}{{Tr}}

\renewcommand{\b}{{\bs{b}}}
\newcommand{\e}{\bs {e}}

\renewcommand{\P}{\bs{P}}

\newcommand{\bs}{\boldsymbol}

\renewcommand{\v}{\bs{v}}
\newcommand{\x}{\bs {x}}
\newcommand{\X}{\bs {X}}
\newcommand{\y}{\bs {y}}
\newcommand{\Y}{\bs {Y}}
\newcommand{\W}{\bs {W}}

\newcommand{\z}{\bs{z}}

\newcommand{\0}{\bs{0}}

\newcommand{\abs}[1]{\ensuremath{\left| #1  \right| }}
\newcommand{\vect}[1] {\ensuremath{\left(\begin{array}{c} #1 \end{array} \right)}} 
\newcommand{\mat}[1] {\ensuremath{ \left(\begin{array} #1 \end{array} \right)}} 
\newcommand{\branchdef}[1] {\ensuremath{ \left\{\begin{array}{rl} #1 \end{array} \right. }} 
\newcommand{\sbra}[1] {\ensuremath{ \left[ #1\right]}} 
\newcommand{\rbra}[1]{\ensuremath{\left( #1 \right)}} 
\newcommand{\bra}[1]{\ensuremath{\left\{ #1 \right\}}} 
\newcommand{\floor}[1]{\ensuremath{\left\lfloor #1 \right\rfloor}}

\newcommand{\half}{\frac{1}{2}}

\newcommand{\iid}{independent identically distributed (i.i.d.) }

\newcommand{\ra}{\rightarrow}

\newcommand{\eq}[1]{\(#1\)}

\newcommand{\kdc}{\(k\)-disjoint-clique }

\begin{document}

\title{Exact Clustering of Weighted Graphs via Semidefinite Programming}
\author{Aleksis Pirinen \thanks{Centre for Mathematical Sciences, Lund University, Lund, Sweden, aleksis@maths.lth.se  }
	\and
	 Brendan P.W.~Ames
	\thanks{ Department of Mathematics,
	The University of Alabama,
	Tuscaloosa, Alabama, AL 35487-0350, USA, bpames@ua.edu}
}

\maketitle
\begin{abstract}%
As a model problem for clustering, we consider the densest $k$-disjoint-clique problem of partitioning a weighted complete
graph into
$k$ disjoint subgraphs such that the sum of the densities of these subgraphs is maximized.
We establish that such subgraphs can be recovered from the solution of a particular
semidefinite relaxation with high probability if the input graph is sampled from a distribution of clusterable graphs.
Specifically, the semidefinite relaxation is exact if the graph
consists of \(k\) large disjoint subgraphs, corresponding to clusters, with weight concentrated within these subgraphs,
plus a moderate number of nodes not belonging to any cluster.
Further, we establish that if noise is weakly obscuring these clusters, i.e, the between-cluster edges are assigned
very small weights,
then we can recover significantly smaller clusters.
For example, we show that in approximately sparse graphs, where the between-cluster weights tend to zero as the size $n$ of the graph tends to infinity, we can recover clusters of size polylogarithmic in $n$
under certain conditions on the distribution of edge weights.
Empirical evidence from
numerical simulations is also provided to support these theoretical phase transitions to perfect recovery of the cluster structure.
\end{abstract}


\allowdisplaybreaks

\section{Introduction} \label{intro}
\emph{Clustering} is a fundamental problem in machine learning and statistics, focusing on the identification and classification
of groups, called \emph{clusters}, of similar items in a given data set. Clustering is ubiquitous, playing a prominent role
in varied fields such as computational biology, information retrieval, pattern recognition, image processing and computer vision,
and network analysis.
This problem is inherently ill-posed, as the partition or clustering of any given data set will depend heavily on how we
quantify similarity between items in the data set and how we characterize clusters; it is not outside the realm of possibility to
have two drastically different clusterings of the same data if two different similarity metrics are used in the clustering process.
Regardless of the similarity metric used, clustering is a combinatorial optimization
problem at its core: given data, identify a partition or labeling of the data (approximately) maximizing some
measure of quality of the clustering.
Due to the difficulties inherent with optimization over discrete sets, many popular approaches for clustering involve
the approximate solution of an NP-hard combinatorial optimization problem; for example, the spectral clustering heuristic
for the normalized cut problem \citep{dhillon2004kernel, ng2002spectral}, the convex relaxation
approaches for the correlation clustering problem \citep{mathieu2010correlation},
robust principal component analysis \citep{chen2014clustering, oymak2011finding},
and the densest $k$-disjoint-clique problem \citep{ames2014convex, ames2014guaranteed}, among many others.

In spite of the inherent intractability of clustering, 
many recent analyses have established that if
data is sampled from some distribution of clusterable data, then one can efficiently recover the underlying
cluster structure using a variety of clustering algorithms.
In particular, the recent results of \citet{ abbe2016exact,  ailon2013breaking, ames2014convex, ames2014guaranteed, amini2014semidefinite, cai2015robust, chen2014clustering, chen2014improved, chen2014statistical, guedon2015community, hajek2015achieving,
lei2015consistency,  mathieu2010correlation,nellore2013recovery, oymak2011finding,rohe2011spectral, qin2013regularized, vinayak2014sharp} all establish sufficient conditions  under which we can expect to identify the latent cluster structure efficiently.
Most of these results assume that the similarity structure of the data can be
modeled as a graph sampled from some
generalization of the
\emph{stochastic block model} proposed by \citet{holland1983stochastic}. In this model, the nodes of the graph, called the \emph{similarity graph} of the data, are associated with the items in the data set. An edge is drawn between two items with fixed probability $p$ if the corresponding items belong to the same cluster, and with fixed probability $q < p$ if the corresponding items belong to different clusters. 
Under this block model, the analyses cited above establish
 that the block structure of the data can be recovered in polynomial-time
 with high probability
provided that the smallest cluster in the data is sufficiently large, typically larger than \( \tilde c \sqrt{n} \), where \(n\)
denotes the number of items in the data (and nodes in the similarity graph) and \(\tilde c \) is a polylogarithmic factor in \(n\)
depending on \(p - q \).

Although valuable in establishing sufficient conditions for data to be clusterable, these results are not immediately
applicable to data sets seen in many applications, particularly those arising from the analysis of social networks.
For example, statistical analysis of social networks suggests that communities, playing the role of clusters,
tend to be limited in size to several hundred users, while the networks themselves can contain thousands, if not millions or even billions,
of users \citep{leskovec2008statistical, leskovec2009community}.
However, the recent analyses of \cite{chen2014clustering,chen2014statistical, guedon2015community,jalali2015relative, rohe2012highest}, among others, suggest that these clusterability results are overly
conservative with respect to the size of clusters we can expect to recover in polynomial-time.
Specifically, these analyses allow the edge probabilities $p$ and
$q$ to vary with $n$, and investigate how the size of the smallest cluster
that can be recovered depends on the relative scaling of $p,q$ and $n$.
In this case, the data is often assumed to be sampled from a \emph{sparse} generalized stochastic block model where the parameters
\(p\) and \(q\) governing edge formation are functions depending on the number of items \(n\) and one
or both tends to \(0\) as \(n \ra \infty\). In the case where \(p \) tends to \(0\) much more slowly than \(q\), the noise
obscuring the block structure is significantly weaker than in the dense graph case (where \(p\) and \(q\) are assumed fixed).
Here, sparsity refers to the fact that graphs
generated according to the block model contain very few edges
between
clusters with high probability when $n$ is large,
and not that the graph itself is sparse in the sense that the nodes
have small average degree.
In this 
case, it has been shown that clusters significantly smaller than
\(\sqrt{n}\) can be recovered efficiently;
specifically, several methods have been shown to recover clusters with size polylogarithmic in \(n\)
under certain assumptions on the probability functions \(p\) and \(q\)  \citep[see][]{chen2014clustering,chen2014statistical, guedon2015community,rohe2012highest}.
We should note that these results provide evidence of a computational limit for cluster recovery; that is, these results establish that clusters can be recovered in a computationally efficient way if the underlying data satisfies certain sufficient conditions. We should note further that the lower bounds on cluster size given by these sufficient conditions typically do not match information-theoretic limits; it is well-known that it is possible to identify clusters of size on the order of $\log n$ in certain settings, however, no polynomial-time algorithms are known to do so (\citealt{chen2014statistical, hajek2015achieving} provide further details).

The primary contribution of this paper is an analysis establishing similar clusterability results for a particular
convex relaxation of the clustering problem.
That is, we present an analysis establishing the following theorem, which provides conditions for perfect recovery of the underlying cluster structure from
the solution of a particular semidefinite program.
As an immediate corollary, the theorem establishes that
one may identify clusters
as small as \( \Omega(\log n)\),
i.e.,  there exists constant $c$ such that the size of the smallest cluster recoverable cluster is bounded below by \(c \log n \) for sufficiently large \(n\),
with high probability
if the data is sampled from the sparse block model described above
for particular choices of $p$ and $q$.
Here, we say that an event occurs \emph{with high probability (w.h.p.)} if the event occurs with probability
tending polynomially to $1$ as $n \ra \infty$.

\begin{theorem} \label{thm: simple}
	Suppose that the \(n\)-node graph \(G = (V,E)\) is sampled from the generalized stochastic block model,
	with \(k\) disjoint blocks, in-cluster edge probability \(p\), and between-cluster edge probability \(q\).
	Let \(\bs A \in \R^{n\times n}\) denote the adjacency matrix of \(G\)
	and let \(\hat r\) and \(\tilde r\) denote the cardinality of the smallest and largest clusters, respectively,
	in the block model for \(G\).
	Then there exists constants $c_1, c_2, c_3 > 0$ such that the columns of the optimal solution \(\X^*\) of
	the semidefinite program
	\begin{equation*} \label{eq: SDP}
	\max_{\X \in \Sigma_+^n} \bra{ \tr(\bs A \X) : \X \e \le \e, \tr(\X) = k, \X \ge \0 }
	\end{equation*}
	are scalar multiples of the characteristic vectors of the clusters in our underlying block model
	with high probability
	if
	\[
		p - q \ge c_3 \max \bra{ \sqrt{\frac{\tilde \sigma^2 \log n}{\hat r}}, \frac{\log n}{\hat r} },
	\]
	where $\tilde \sigma^2 = \max\{{p(1-p)}, q(1-q) \}$, and
	\begin{equation*}
		(p-q) \hat r \ge c_1 \max \bra{ \sqrt{ q(1-q) n },\sqrt{ \log n}  }
		+ c_2 \max \bra{ \sqrt{p(1-p) \tilde r }, \sqrt{\log n} }.
	\end{equation*}
	Moreover, in this case, every characteristic vector of a cluster in the block model is a
	scalar multiple of at least one column of \(\X^*\).
\end{theorem}

Here, the characteristic vector of a set
\(S \subseteq \{1,2,\dots, n\}\) is the vector \( \x \in \{0,1\}^n \) 
with \(i\)th element
\[
x_i = \branchdef{ 1, & \mbox{if } i \in S \\ 0, &\mbox{otherwise.}}
\]
In Theorem~\ref{thm: simple}, \(\tr(\X) \) denotes the trace of the matrix \(\X\),
\(\e\) denotes the all-ones vector of appropriate dimension,
the notation \(\X \ge \0\) indicates that the entries of \(\X\) are nonnegative,
and \(\Sigma_+^n\) denotes the cone of \(n\times n\) symmetric positive semidefinite matrices.

Note that if \(G\) is sampled from the \emph{dense} block model, i.e., \(p,q\) are independent of \(n\), then Theorem~\ref{thm: simple} suggests that we have exact recovery if \( \hat r \ge c \sqrt{n}\) with high probability, where \(c\) is a constant
depending on \(p, q\); this bound matches that established by \cite{ames2014guaranteed} (among many others) up to
constant terms.
On the other hand, when \(G\) is sampled from the \emph{sparse} block model,
we see that Theorem~\ref{thm: simple} suggests that we may have perfect recovery of significantly smaller clusters.
For example, suppose that \(p=1\) is fixed and \(q = \log n/ n\).
Then we have exact recovery with high probability if
the smallest cluster has size \(\hat r = \Omega( \log n)\);
see the discussion following Theorem~\ref{thm: rec}.


We will show that analogous phenomena occur in what we will call
\emph{approximately} sparse graphs.
In many practical applications, the expectation that we have a binary labeling
indicating whether any pair of items in a given data set are similar
or dissimilar is unrealistic. However, it is often possible to
describe the \emph{level} of similarity between any two items using some affinity
function based on distance between the items in question.
For example, we could consider the discrepancy in pixel intensity and geographic
location in image segmentation applications or Euclidean distance between
two items represented as vectors in a Euclidean space (or some other vector
space with corresponding norm).
In this case, we can summarize the pairwise similarity relationships within our
data using
a weighted graph, called a \emph{weighted similarity graph}.
Specifically, given a data set with affinity function \(f\),
the weighted similarity graph is the weighted complete graph with nodes
corresponding to the items in the data set, and edge weight \(w_{ij}\)
between nodes \(i\) and \(j\) given by the value of \(f(i,j)\).
Clearly, this contains the similarity graphs discussed earlier as a
special case where \(w_{ij} = 1\) if items \(i\) and \(j\) are known
to be similar and \(w_{ij} = 0\) otherwise; note that we assume that we have an undirected graph with
symmetric adjacency matrix.

We can generalize 
the stochastic block model in an identical fashion.
We assume that items in the same cluster are significantly
more similar than pairs of items in different clusters.
This corresponds to edge weights within clusters being larger, on average,
than edge weights between clusters.
This motivates the following random graph model,
which we will call the \emph{planted cluster model}.
Let \(G = (V,\W)\) be the weighted complete graph whose node set
represents the items in some data set containing \(k\) clusters and (potentially)
some nodes that will not be assigned to a cluster.
For each pair of nodes \(u,v \) in the same cluster $C_\ell$,
$\ell \in \{1,2,\dots, k\},$
we randomly sample
edge weight \(w_{uv} \ge 0\), and \(w_{vu}\) by symmetry,
from some probability distribution \(\Omega_\ell\)
with mean \(\alpha_\ell \ge \alpha > 0\). If \(u \in C_i\),
\(v \in C_j\), where $i \neq j$, i.e., $u,v$ do not belong to the same
cluster, we sample \(w_{uv} = w_{vu} \ge 0\) from a different probability distribution
\(\Omega_{ij}\) with mean \( \beta_{ij} \le \beta \in [0,\alpha) \).
Note that this model contains the generalized stochastic block model discussed
earlier as a special case when \(\Omega_\ell \) and \(\Omega_{ij}\)
are Bernoulli distributions with probabilities of success
\(p_\ell = p\) and \(q_{ij} = q\), respectively.

It was shown by \cite{ames2014guaranteed} that if \(G = (V, \W)\) is sampled
from the
planted cluster model
with minimum cluster
size at least \(c\sqrt{n}\) in the homogeneous
case where all within-cluster edges are i.i.d.~with mean $\alpha$
and all between-cluster edges are i.i.d.~with mean $\beta$, where \(c\) is a constant
depending on \(\alpha\) and \(\beta\), then we can recover the clusters
from the optimal solution of the semidefinite program
\begin{equation} \label{eq: cluster SDP}
\max_{\X \in \Sigma_+^n} \Big\{ \tr(\W \X): \X \e \le \e,
\tr(\X) = k, \X \ge \0 \Big\}
\end{equation}
with high probability, where \(k \) is the number of clusters in the graph.
We will show that these results can be strengthened to establish that much
smaller clusters can be recovered in the presence of \emph{approximately}
sparse noise. That is, we will
see that if the between-cluster edge weights have expectation
\(\beta\) and variance \(\sigma_2^2\) approaching zero  sufficiently quickly as \(n \ra \infty\), then we may recover
clusters containing as few as \(\Omega(\log n)\) nodes with high probability. 
We will derive the semidefinite program \eqref{eq: cluster SDP}
as a relaxation of a particular model problem for clustering in Section~\ref{sec: kdc}
and formally state our recovery guarantees  in Section~\ref{sec: rec};
we will see that these results immediately specialize to those stated in Theorem~\ref{thm: simple}
for the semidefinite program~\eqref{eq: SDP}.

\section{Semidefinite Relaxations of the Densest k-Disjoint Clique Problem}

In this section, we derive a semidefinite relaxation for the densest \(k\)-disjoint clique problem
and present an analysis illustrating a sufficient condition ensuring that this relaxation
is exact. This problem will act as a model problem for clustering and we will see that
we should expect
to accurately recover the underlying cluster structure if the given data
satisfies this sufficient condition.

\subsection{The Densest k-disjoint Clique Problem}
\label{sec: kdc}
We begin by deriving a heuristic for the clustering problem based on semidefinite relaxation
of the densest disjoint clique problem.
A similar discussion motivating the relaxation was originally presented by \cite{ames2014guaranteed}; we repeat it here for completeness.
Let $K_n=(V, \W)$ be a weighted complete graph with vertex set
$V=\{1,2,\dots,n\}$ and nonnegative edge weights $w_{ij} \in [0,1]$ for all $i,j \in V$.
Given a subgraph \(H\) of \(K_n\), the \emph{density} \(d_H\) of \(H\)
is the average edge weight incident at a vertex in \(H\):
\[
d_H = \sum_{ij \in E(H)} \frac{w_{ij}}{|V(H)|}.
\]
If we assume that \(K_n\) is the similarity graph of some data set consisting
of \(k\) disjoint clusters and that weight
is concentrated more heavily on within-cluster edges than between-cluster
edges, then we may cluster this data set by finding the set of \(k\) disjoint
subgraphs, corresponding to these clusters, with maximum density;
we call this problem the \emph{densest \(k\)-partition problem}.
\citet{peng2007approximating} established that the {densest \(k\)-partition problem is NP-hard. Moreover, this partition model excludes the
	case where some items in the data set do not naturally associate
	with any of the clusters in the data.
	To simultaneously motivate a convex relaxation of the densest \(k\)-partition
	problem and address the inclusion of nodes that do not naturally belong to clusters, we consider the \emph{densest
		\(k\)-disjoint clique problem}.

	Given a graph \(G=(V,E)\), a \emph{clique} of \(G\) is a pairwise
	adjacent subset of \(V\). That is, \(C \subseteq V\)
	is a clique of \(G\) if \(ij\in E\) for every pair of nodes \(i,j \in C\)
	or, equivalently, the subgraph \(G(C)\) induced by \(C\) is
	{complete}.
	We say that \(H\) is a \emph{\(k\)-disjoint-clique} subgraph of \(K_n\)
	if \(V(H)\) consists of \(k\) disjoint cliques, i.e., \(H\) is the union
	of \(k\) disjoint complete subgraphs of \(K_n\).
	The \emph{densest \(k\)-disjoint-clique problem} seeks a \(k\)-disjoint-clique
	subgraph \(H^*\) maximizing the sum of the densities of the disjoint
	complete subgraphs comprising \(H^*\).
	Note that if we add the additional constraint that each node in \(K_n\) belongs
	to exactly one \kdc subgraph in \(K_n\), then
	the densest \kdc problem becomes the densest \(k\)-partition problem.
	However, in general, the densest \kdc problem allows an assignment of nodes
	to clusters, represented by the disjoint cliques, which excludes
	some nodes. For example, if such nodes are present in the data, they would not be assigned to a
	cluster by the optimal \kdc subgraph.

	The complexity of the densest \kdc problem is unknown;
	in particular, no polynomial-time algorithm for its solution is known.
	To address this potential intractability, we will attempt to approximately solve the \kdc problem by convex
	relaxation.
	Suppose that \( \v_1, \dots, \v_k \) are the characteristic vectors
	of a set of disjoint cliques \(C_1, C_2, \dots, C_k\)
	forming a \kdc subgraph of \(K_n\).
	Using this notation, the density of the complete subgraph induced
	by \(C_i\) is equal to
	\[
	d_{G(C_i)} = \sum_{u, v \in C_{i}} \frac{w_{uv}}{|C_i|}
	= \frac{\v_i^T \W \v_i}{\v_i^T \v_i}.
	\]
	If we let \(\P\) be the \(n \times k\) matrix with \(i\)th column equal to \(\v_i/ \|\v_i\|\), where \(\|\cdot\| = \|\cdot\|_2\)
	denotes the standard Euclidean norm, then
	it is easy to see that
	\[
	\sum_{i=1}^k d_{G(C_i)} = \tr(\P^T \W \P).
	\]
	We call such a matrix \(\P\) a \emph{normalized \(k\)-cluster matrix} and denote the set of normalized \(k\)-cluster
	matrices of the vertex set \(V\) by \(ncm(V,k)\).
	It follows that the densest \kdc problem may be formulated as
	\begin{equation} \label{eq: kdc prob}
	\max \bra{ \tr(\P^T \W \P) : \P \in ncm(V,k) }.
	\end{equation}
	Again, the complexity of \eqref{eq: kdc prob} is unknown, however, the maximization of quadratic functions
	subject to combinatorial constraints is known to be NP-hard.

	 A process for relaxation of \eqref{eq: kdc prob} using matrix lifting is described by \cite{ames2014guaranteed};
	a similar relaxation technique was applied by
	\cite{ames2011nuclear, ames2014convex}
	and \cite{ames2015guaranteed}.
	In particular, each proposed cluster \(C_i\), with characteristic vector \(\v_i\), corresponds to the rank-one symmetric matrix
	\[
	\X^{(i)} =  \frac{ \v_i \v_i^T}{ \v_i^T \v_i}.
	\]
	It is easy to see that the density of \(G(C_i)\) is equal to
	\[
	d_{G(C_i)}= \frac{ \v_i^T \W \v_i }{\v_i^T \v_i }  = \tr( \W \X^{(i)}) .
	\]
	Moreover, each of the matrices \(\X^{(i)}\) has row and column sums equal to either 0 or 1, and trace equal to 1.
	Finally,
	for each proposed clustering \(C_1, \dots, C_k\), the corresponding rank-one matrices are orthogonal in the trace inner
	product, due to the orthogonality of the characteristic vectors of the corresponding disjoint clusters.
	Thus, the matrix
	\begin{equation} \label{eq: cluster X}
	\X = \sum_{i=1}^k \X^{(i)} = \sum_{i=1}^k \frac{ \v_i \v_i^T}{ \v_i^T \v_i}
	\end{equation}
	has rank equal to \(k\).
	This suggests that we may relax \eqref{eq: kdc prob} as the rank-constrained semidefinite program
	\begin{equation} \label{eq: kdc rank sdp}
	\max_{\X \in \Sigma^n_+} \bra{ \tr(\W\X): \X \e \le \e, \rank \X = k, \tr \X = k, \X \ge \0 }.
	\end{equation}
	The relaxation \eqref{eq: kdc rank sdp} can be relaxed further to a semidefinite program
	by omitting the nonconvex rank constraint:
	\begin{equation} \label{eq: kdc sdp}
	\max_{\X \in \Sigma^n_+} \bra{ \tr(\W\X): \X \e \le \e, \tr \X = k, \X \ge \0}.
	\end{equation}
	We should note that the semidefinite program \eqref{eq: kdc sdp} is remarkably similar to the
	semidefinite relaxation of the minimum sum of squared distance partition of \cite{peng2007approximating}
	and the semidefinite relaxation of the maximum likelihood estimate of
	the stochastic block model considered by \cite{amini2014semidefinite}, among others,
	although our relaxation approach differs slightly from that used in these two papers.

\subsection{Block Models and Recovery Guarantees}
\label{sec: rec}
Given a set of {clusterable} data or, more accurately, a {clusterable} graph representation of data, \cite{ames2014guaranteed} established that one can recover the underlying cluster structure from the optimal solution of the semidefinite program \eqref{eq: kdc sdp}.
Specifically, it is assumed that data with strong cluster structure should correspond to similarity graphs
with heavy weight assigned to edges within clusters, relative to that between cluster edges.
This corresponds to pairs of items within clusters being significantly more similar than pairs of items in
different clusters. This motivates the following block model.

Let \(H^*\) be a \(k\)-disjoint-clique subgraph of \(K_n = (V,\W)\) with vertex set composed
of the disjoint cliques \(C_1, \dots, C_k\) and
let \(\Sigma^n\) denote the set of all \(n\times n\) symmetric matrices.
We consider weight matrices \(\W = [w_{ij}] \in \Sigma^n\)
with entries sampled independently from one of two probability distributions \(\Omega_1, \Omega_2\)
as follows. 
\begin{itemize}
	\item
	For each \(i=1, \dots, k\) and each \(u,v \in C_i\),
	we sample \(w_{uv} = w_{vu}\) from a distribution \(\Omega_1\) such that
	\[
	\E[w_{uv}] = \E [w_{vu}] = \alpha, \hspace{0.25in} 0 \le w_{uv} \le 1,
	\]
	for fixed \(\alpha \in (0,1]\).
	\item
	For each remaining 	edge \(uv\),
	\eq{u\in C_i, v \in C_j}, we sample the edge weight \(w_{uv} = w_{vu}\)
	from a second distribution \(\Omega_2\) such that
	\[
	\E[w_{uv}]  = \E[w_{vu}] = \beta, \hspace{0.25in} 0 \le w_{uv} \le 1,
	\]
	for fixed \(\beta \in [0, \alpha) \)
	if ${1 \le i, j \le k}$ or $i=j = k+1$,
	and \(\E[w_{uv}] = \beta/2\) otherwise.
\end{itemize}

We should note that the assumption that the entries of \(\W\) are bounded between \(0\) and
\(1\) is made for simplicity in the statement and proof of our main result;
analogous recovery guarantees hold if we assume that random variables
sampled according to \(\Omega_1\) and \(\Omega_2\)
are bounded and nonnegative with high probability.
We say that such random matrices \(\W\) are sampled from the \emph{planted cluster model}.
Note that if \(\W\) is sampled from the planted cluster model, then weight is concentrated
on within-cluster edges (in expectation).
This provides a natural generalization of the stochastic block model.
Indeed, the stochastic block model corresponds to the planted cluster model
in the special case that \(\Omega_1\) and \(\Omega_2\) are Bernoulli distributions with probabilities of
success \(p\) and \(q\), respectively.
\cite{ames2014guaranteed} established the following theorem, ensuring recovery of the planted cliques
\(C_1, \dots, C_k\) from the optimal solution of \eqref{eq: kdc sdp} under the planted cluster model \citep[see][Theorem 2.1]{ames2014guaranteed}.

\begin{theorem}
	\label{thm: old rec}
	Suppose that the vertex sets \(C_1, \dots, C_k\) define a \kdc subgraph \(H^*\) of the
	\(n\)-node weighted complete graph \(K_n(V, \W)\) and let \(C_{k+1} := V \setminus \rbra{ \cup^k_{i=1} C_i } \).
	Let \(r_i := |C_i|\) for all \(i=1,\dots, k+1\)
	and let \( \hat r = \min_{i=1,\dots,k} r_i.\)
	Let \( \W \in \Sigma^n\) be a random symmetric matrix sampled from the planted cluster model
	according to distributions \(\Omega_1\) and \(\Omega_2\) with means \(\alpha\) and \( \beta\), respectively,	satisfying
	\[
	\gamma = \gamma(\alpha, \beta, r):=
	\alpha - \beta > 0.
	\]
	Let \(\X^*\) be the feasible solution of \eqref{eq: kdc sdp} corresponding to \(C_1,\dots, C_k\)
	defined by \eqref{eq: cluster X}.
	Then there exist scalars \(c_1, c_2, c_3 > 0\) such that if
	\begin{equation*} \label{eq: old rec}
	c_1 \sqrt{n} + c_2 \sqrt{k r_{k+1}} + c_3 r_{k+1} \le \gamma \hat r,
\end{equation*}
	then \( \X^* \) is the unique optimal solution of \eqref{eq: kdc sdp},
	and \( H^*\) is the unique maximum density \kdc subgraph of \(K_n\)
	with probability tending exponentially to \(1\) as \(\hat r \ra \infty\).
\end{theorem}

In contrast to Theorem~\ref{thm: simple}, the result of Theorem~\ref{thm: old rec} implies that we can have perfect recovery if the graph contains a small number of nodes that shouldn't be assigned to any of the planted clusters.
Each potential edge from each of these nodes to any other node is added independently to the graph with probability $q$, so that each node in $C_{k+1}$ has roughly the same number of neighbours in each cluster block. This implies that such a node is not assigned to any of the planted clusters because it is weakly associated with all of the planted clusters. It is important to note that this edge assignment is performed randomly and not deterministically by an adversary attempting to obscure the cluster structure present in the graph.
%
We present a new analysis that improves upon the recovery guarantee of Theorem~\ref{thm: old rec} in two ways.
First, the hypothesis of Theorem~\ref{thm: old rec} assumes that between-cluster and within-cluster edge weights are~i.i.d. We consider the more general heterogeneous case constructed as follows:
\begin{itemize}
	\item For each $u \in C_i$, $v \in C_j$, we sample the edge weight $w_{uv} = w_{vu}$ from distribution $\Omega_{ij}$ with
	\[
	\E[w_{uv}]  = \E[w_{vu}] = \mu_{ij}, \hspace{0.25in}
	\text{Var}[w_{uv}] = \text{Var}[w_{vu}] = \sigma^2_{ij},
	\hspace{0.25in}
	0 \le w_{uv} \le 1.
	\]
\end{itemize}
This forces weights within the same block to be i.i.d., but weight may not be identically distributed in different blocks.

Second, the analysis leading to Theorem~\ref{thm: old rec} assumes that the expectations of \(\Omega_{1}, \Omega_2\)
in the planted cluster model are fixed and that the variances are bounded by \(1\).
We improve upon the recovery guarantee of Theorem~\ref{thm: old rec} by considering the case
where the parameters \(\alpha\) and \(\beta\) depend on the number of nodes \(n\) in the graph.
In particular, our recovery guarantees explicitly depend on the
variances of the distributions \(\Omega_{ij}\), and their scaling with $n$, which will expand the set of graphs
known to be clusterable by \eqref{eq: kdc sdp}. We have the following theorem.

\begin{theorem} \label{thm: rec}
	Suppose that the vertex sets \(C_1,\dots,C_k\) define a \(k\)-disjoint-clique subgraph \(K^*\) of the weighted
	complete graph
	\(K_n = (V,\bs W)\) on \(n\) vertices and let \(C_{k+1} = V \setminus \left(\cup_{i=1}^k C_i \right)\). Let \(r_i = |C_i|\) for all
	\(i=1,\dots,k+1\) and let \(\hat{r} = \min_{i=1,\dots,k} r_i\).
	Let \(\bs{W} \in \Sigma^n\)
	be a random symmetric matrix sampled from the heterogeneous planted cluster model according to distributions \(\{\Omega_{ij} \}\) with expected values \(\mu_{ij} = \mu_{ij}(n)\) and variances
	\(\sigma_{ij}^2 = \sigma_{ij}^2(n)\).
	Let $\tilde \sigma := \max_{q,s} \sigma_{qs}$ and
	$\hat \sigma := \max_{q} \sigma_{qq}$.
	Let \(\X^*\) be the feasible solution to \eqref{eq: kdc sdp} corresponding to \(C_1,\dots,C_k\) defined by
	\eqref{eq: cluster X}.
  Let
  $$
    \gamma := \min_{\substack{q,s=1,2,\dots, k \\ q\neq s}} \bra{\mu_{qq} - \mu_{qs}}.
  $$
  Then there exists scalar $c > 0$ such that if
   \begin{equation}\label{eq:smallest_largest}
 		\gamma \hat r \ge 
 		c \max \bra{ \sqrt{\tilde\sigma^2 n},\;
      \sqrt{\tilde\sigma^2 \hat r \log n},\;
      \sqrt{\hat\sigma^2 k r_{k+1}}, \;
      \sqrt{k r_{k+1} \log n / \hat r}, \;      \mu_{k+1, k+1} r_{k+1}, \; \log n
      %
      },
 	\end{equation}
	then \(\bs{X}^*\) is the unique optimal solution for \eqref{eq: kdc sdp}, and \(K^*\) is the unique maximum density \kdc
	subgraph of \(K_n\) with high probability.
\end{theorem}
%

The \emph{weak assortativity} condition~\eqref{eq:smallest_largest} implies that we have perfect recovery provided that the gap between the cluster block expectation $\mu_{qq}$ and the largest between-cluster block expectation $\mu_{qs}$ is sufficiently large for all clusters $C_q$, $q=1,\dots, k$, relative to the minimum cluster size, number of unassigned nodes $r_{k+1}$, number of clusters,  and edge weight variances.
In the 
Bernoulli case, i.e., within-cluster and between-cluster edges are added independently with probabilities $p$ and $q$, respectively, Theorem~\ref{thm: rec} and, in particular,
\eqref{eq:smallest_largest} establish that we can recover the planted
clusters provided that
\[
	\frac{(p-q)^2}{\tilde\sigma^2} = \frac{(p-q)^2}{\max\{p(1-p), q(1-q)\}} = \Omega\rbra{ \frac{ n }{\hat r^2}}.
\]
This result agrees with the Easy Regime for cluster recovery
proposed by \cite{chen2014statistical}, where a polynomial-time
algorithm exists for exact recovery of the planted clusters, in this case, the solution of the semidefinite relaxation \eqref{eq: kdc sdp}.
One distinct advantage of this result over similar recovery guarantees
is that our model and phase transition are largely parameter free.
For example, \citet{amini2014semidefinite} present an analysis of three
semidefinite relaxations that obtain nearly identical conditions
on $\{\Omega_{ij}\}$ guaranteeing recovery but restrict their
analysis to the case where the clusters are identical in size
or otherwise known and when $\{\Omega_{ij}\}$ are Bernoulli distributions; we should note that
\citet{amini2014semidefinite} consider heterogeneous Bernoulli
distributions where the within-cluster and between-cluster probabilities
of adding an edge vary across clusters.
Similarly, \citet{chen2014statistical} and \citet{jalali2015relative} give identical
conditions for recovery
(up to constants and logarithmic terms)
in the Bernoulli case
to those in Theorem \ref{thm: rec} for
semidefinite relaxations that require the sizes of the clusters to
be used as input parameters (or all clusters to have identical size), neither
of which are realistic assumptions in practice.
In contrast, our approach achieves this recovery guarantee
using \emph{only the desired number of clusters as a parameter}.
Further, our guarantee extends to the general weighted case where
the vast majority of existing recovery guarantees for stochastic
block models are restricted to the Bernoulli case.

It is important to note that tighter recovery guarantees than those provided by Theorem~\ref{thm: rec} are known for specific problem settings. This is a natural consequence of the more general framework of our analysis.
For example, \cite{yan2017provable} studies a convex relaxation for cluster recovery in the Bernoulli (unweighted) case. The main theorem of this article establishes conditions for perfect recovery that allow larger clusters to have higher variance, although specialized for the unweighted case. Moreover, \cite{yan2017provable} consider the use of a tuning parameter to allow recovery without knowledge of the number of clusters $k$. On the other hand, the results of \cite{amini2014semidefinite,jalali2015relative} also provide tighter recovery guarantees but require knowledge of cluster sizes.  The key contribution of this work is the presentation of a recovery guarantee that extends to the weighted case without strict assumptions regarding input parameters, as well as the first-order method for solution of~\eqref{eq: cluster SDP} discussed in detail in Section~\ref{sec: num}.

To further illustrate the consequences of Theorem~\ref{thm: rec}, we consider
several examples. In each, we assume that the graph is
generated in the homogeneous setting where
within-cluster weights are i.i.d.~according to $\Omega_1$ with mean $\alpha$ and variance $
\sigma_1^2$, and
between-cluster weights are i.i.d.~according to
$\Omega_2$ with mean $\beta < \alpha$ and variance $\sigma_2^2$.

\subsubsection{The Dense Case}

When \( \alpha, \beta \) are fixed, we obtain the same recovery guarantee as before, up to constants and logarithmic
terms:  we have exact recovery w.h.p.~if \(\hat r \ge \tilde c_1 \sqrt{n} \) and \(\hat r \ge \tilde c_2 r_{k+1} \)
for some constants \(\tilde c_1\) and \(\tilde c_2\) depending on \(\Omega_1\) and \(\Omega_2\).
Indeed, each of the pointwise maximums in the first  three terms
of \eqref{eq:smallest_largest} is bounded above by $O(\sqrt{n})$
since $\tilde r \le n$, and $k r_{k+1} = O(n)$ if $r_{k+1} = O(\hat r)$.

\subsubsection{The Sparse Case}


On the other hand, if noise in the form of between-cluster
edge-weight is small, then we
should expect to be able recover
much smaller clusters.
For example, suppose that \(\Omega_2\) is the Bernoulli distribution with probability of adding an edge \(q\) and that \(\Omega_1\) is the Bernoulli distribution with probability of adding an edge \(p=1\) (the assumption that \(p=1\) is for the sake of simplicity in this example and we can expect analogous recovery guarantees for
any \(p\) tending slowly enough to \(0\)).
Assume further that
\(
q (1-q)  \le {\log n }/{n}.
\)
Finally, again for simplicity, assume that we have \(k\) equally sized clusters of size \(\hat r = n/k\)
and  (\(r_{k+1} = 0\)).
In this case, \eqref{eq:smallest_largest}
holds if
$$
\gamma \hat r \ge c \log n = c \max \bra{\sqrt{\log n},\; \sqrt{\frac{ \hat r \log^2n}{n}},\;\log n}
\ge c \max \bra{\sqrt{\tilde\sigma^2 n}, \; \sqrt{\tilde\sigma^2 r \log n},\; \log n},
$$
since $\tilde \sigma^2  = \max\{p(1-p), q(1-q)\} = q(1-q) \le \log n / n$
and the terms involving \(r_{k+1}\) and \(1-p\) are equal to zero.
This implies that we have exact recovery of the planted clusters
w.h.p.~provided $\hat r =\Omega(\log n)$. This exceeds the
state of the art recovery bound
of $\hat r = \Omega(\sqrt{\log n})$ established in \cite{jalali2015relative}
by a factor of $\sqrt{\log n}$.
However, the convex relaxation proposed by
\cite{jalali2015relative} requires knowledge
of $\sum_{i=1}^k r_i^2$, which is often an unrealistic
expectation in practice;
in contrast, our approach only requires knowledge
of the number of clusters $k$ present in the data.
Further,  the requirement $\hat r = \Omega(\log n)$
is enforced by the gap inequality \eqref{eq:smallest_largest},
which itself is a consequence of the use of the Bernstein inequality
to establish certain dual variables are nonnegative
in the proof of Theorem~\ref{thm: rec} (see Section~\ref{sec: nonneg}
for more details). It may be possible to improve this bound
to $\hat r = \Omega(\sqrt{\log n})$ with improved concentration
inequalities
 but it is unclear
what form these improvements may take.

\subsubsection{The Planted Clique and Sparsest Subgraph}

In the special case when $k=1$ and $\Omega_1$ and $\Omega_2$
are Bernoulli distributions, the planted cluster model
specializes to the planted clique model considered in
\cite{ames2011nuclear} and \cite{ames2015guaranteed}.
In this case, \eqref{eq:smallest_largest}
suggests that we can recover a planted clique (in the dense
case) of size $r_1 = \Omega \rbra{ \max \bra{\sqrt{n}, r_2 } } = \Omega \rbra{ \max \bra{\sqrt{n}, n - r_1 } }$.
This recovery guarantee is far more conservative than those
provided by \cite{ames2011nuclear} and \cite{ames2015guaranteed},
among others,
which establish that a planted clique
of size $\Omega(\sqrt n)$ can be recovered from the optimal solution
of a particular nuclear norm relaxation of the maximum clique problem.

Unfortunately, it appears that this lower bound restricting the size
of a recoverable planted clique
to a constant multiple of the number of nonclique vertices
is tight.
For example, let $p$ and $q$ be the probabilities of adding an edge
given by $\Omega_1$ and $\Omega_2$.
Then the expected value of the proposed solution $\X^*$
in \eqref{eq: kdc sdp}
is equal to
$$
	\E [ \tr(\W \X^*)] = \frac{1}{\hat r} \sum_{i \in C_1} \sum_{j \in C_1}
		\E[w_{ij} ] = p \hat r.
$$
On the other hand, the solution $\frac{1}{n} \e\e^T$ is also
feasible for \eqref{eq: kdc sdp} with expected objective value
$$
	\E \sbra{ \frac{1}{n} \tr(\W \e\e^T) }
	 \ge \hat c q n,
$$
for some constant $\hat c$.
This implies that the proposed solution is suboptimal if
$p \hat r < \hat c q n$, which holds unless $\hat r \ge \hat c (q/p) n$.
We will see that the realized values of these sums
are concentrated near their expectations and thus we cannot
reasonably expect to recover planted clusters with unassigned nodes significantly
outnumbering the smallest cluster. This implies that we cannot
recover planted cliques of size $\omega(n)$ by maximizing
density of a complete subgraph because the planted clique
is not the index set of the densest such graph; in this case, the
entire graph is a denser
complete subgraph, as measured by average vertex degree,
in expectation.

\section{Derivation of the Recovery Guarantee}
\label{sec: proof}
In this section, we show that if the hypothesis of Theorem~\ref{thm: rec} is satisfied
then the solution \(\X^*\) constructed according to \eqref{eq: cluster X}
is optimal for \eqref{eq: kdc sdp} and the corresponding \kdc subgraph has maximum density.
In particular, we will show that \(\X^*\) satisfies
the following sufficient condition for optimality of a feasible solution of \eqref{eq: kdc sdp}
\citep[see][Theorem 4.1]{ames2014guaranteed}.

\renewcommand{\S}{\bs{S}}
\begin{theorem}
	\label{T:  w KKT conditions}
	Let $\X$ be feasible for \eqref{eq: kdc sdp} and suppose that
	there exist some $\tau \in \R$, $\bs\lambda \in \R^n_+$, $\bs\Xi \in \R^{n\times n}_{+}$ and $\S \in \Sigma^n_+$ such that
	\begin{align}
	- \W + \bs\lambda \e^T + \e \bs\lambda^T - \bs\Xi + \tau \bs I &= \S \label{E: w dual feas} \\
	\bs\lambda^T (\X\e - \e)  &= 0  \label{E: w CS rowsum} \\
	\tr(\X  \bs\Xi) &= 0 \label{E: w CS nonneg} \\
	\tr (\X \S)  &= 0. \label{E: w CS sdp}
	\end{align}
	Then $\X$ is optimal for \eqref{eq: kdc sdp}.
\end{theorem}

Theorem~\ref{T: w KKT conditions} is a restriction of the Karush-Kuhn-Tucker optimality conditions
to the semidefinite program \eqref{eq: kdc sdp} \citep[see, for example,][Section 5.5.3]{boyd2009convex}.
The goal of this section is to establish that we can construct dual variables $\tau \in \R$, $\bs\lambda \in \R^n_+$, $\bs\Xi \in \R^{n\times n}_{+}$ and $\S \in \Sigma^n_+$
which satisfy the hypothesis
of Theorem~\ref{T: w KKT conditions} with high probability if the weight matrix \(\W\) is
sampled from a distribution of clusterable block models.
To motivate our proposed choice of dual variables, we note that
the complementary slackness condition \(\tr(\X\S) = 0\) holds if and only if \(\X\S = \bs 0\) under the assumption that
both \(\X\) and \(\S\) are positive semidefinite.
Therefore, the block structure of \(\X\) implies that
each block of \(\S\) corresponding to a cluster block in \(\W\) must sum to zero.

Before we continue with the construction of our dual variables, let us first remind ourselves
of the notation of Theorem~\ref{thm: rec}.
Let $K^*$ be a $k$-disjoint-clique subgraph of $K_n$ with vertex set composed of the disjoint cliques $C_1, \dots, C_k$ of sizes $r_1, \dots, r_k$  and let
$\X^*$ be the corresponding feasible solution of \eqref{eq: kdc sdp} defined by \eqref{eq: cluster X}.
Let $C_{k+1} := V \setminus  (\cup^k_{i=1} C_i)$ and $r_{k+1} := n - \sum_{i=1}^k r_i$ be the size
of  \(C_{k+1}\).
Moreover, let $\hat r := \min_{i=1, \dots, k} r_i$ and
\( \tilde r := \max_{i=1,\dots, k} r_i\)  be the size of the smallest and largest clusters,
respectively.
Let $\W \in \Sigma^n$ be a random symmetric matrix
sampled from the planted cluster model with planted clusters \(C_1, \dots, C_k\) and
remaining nodes \(C_{k+1}\)
according to the distributions \(\{\Omega_{ij}\}\) with means \( \bra{\mu_{ij}} \)
and variances $\{\sigma^2_{ij}\}$.

We now propose a choice of dual variables satisfying the complementary slackness condition \(\X\S = \bs 0\).
Restricting this condition to the blocks \(\X_{C_q, C_q}\) and \(\S_{C_q, C_q}\) of \(\X\) and \(\S\) with
rows and columns indexed by \(C_q\), \(q \in\{1,2,\dots, k\}\), we see that \(\X^*\S=\bs 0\) holds if and only if
\[
\bs 0 = \S_{C_q, C_q} \e = 
\tau \e  + r_q \bs\lambda_{C_q} + (\bs\lambda_{C_q}^T \e) \e - \W_{C_q, C_q} \e,
\]
by the block structure of \(\X^*\);
note that \(\bs\Xi_{C_q, C_q} = \bs 0\) is chosen to satisfy the complementary slackness condition \eqref{E: w CS nonneg}.
Solving this linear system for \(\bs\lambda_{C_q}\)
using the Sherman-Morrison-Woodbury
Formula~\cite[Equation (2.1.4)]{golub2013matrix} gives
\begin{equation} \label{eq: lam q}
\bs \lambda_{C_q} = \frac{1}{r_q} \left( \W_{C_q, C_q} \e - \frac{1}{2}
\left(\tau + \frac{ \e^T \W_{C_q, C_q} \e}{r_q} \right)   \e \right).
\end{equation}
On the other hand, we choose \(\bs\lambda_{C_{k+1}} = \0\)
to satisfy the complementary slackness condition
\eqref{E: w CS rowsum}. 
Next, we use this choice of \(\bs \lambda\) to construct the remaining dual variables.

Fix \(q, s \in \{1,2,\dots, k+1\}\) such that \(q\neq s\).
We will choose \(\bs \Xi_{C_q, C_s} \) so that \(\S_{C_q, C_s}\e = \bs 0\) and \( \S_{C_s, C_q} \e = \bs 0 \).
In particular, we choose
\begin{equation} \label{eq: Xi choice 1}
	\resizebox{0.92\hsize}{!}{$
\bs\Xi_{C_q, C_s} = \rbra{ \frac{1-\delta_{q, k+1}}{2}
\rbra{\mu_{qq} - \frac{\tau}{r_q}} +
\frac{1-\delta_{s, k+1}}{2}
\rbra{\mu_{ss} - \frac{\tau}{r_s}} - \mu_{qs}} \e\e^T
+ \y^{q,s} \e^T + \e (\z^{q,s})^T,
$}
\end{equation}
where the vectors \( \y^{q,s}\) and \(\z^{q,s}\) are unknown vectors parametrizing the entries of \(\bs\Xi_{C_q, C_s}\);
here
\(\delta_{i,j}\) is the Kronecker delta function defined by \( \delta_{i,j} = 1 \) if \(i=j\)
	and \(0\) otherwise.
That is, we choose \(\bs\Xi_{C_q, C_s} \) to be the expected value of
\( \bs \lambda_{C_q} \e^T + \e \bs \lambda_{C_s}^T
- \W_{C_q, C_s} \) plus the parametrizing term \(\y^{q,s} \e^T + \e (\z^{q,s})^T\); the vectors
\( \y^{q,s}\) and \(\z^{q,s}\) are chosen to be solutions of the systems of linear equations given
by the complementary slackness conditions
\(\S_{C_q, C_s}\e = \bs 0\) and \( \S_{C_s, C_q} \e = \bs 0 \).
It is reasonably straight-forward to show that we may choose
\begin{equation} \label{eq: y z}
\y^{q,s} = \frac{1}{r_s} \rbra{ \b_{q,s} - \frac{\b_{q,s}^T \e }{r_q + r_s} \e }
\hspace{0.5in}
\z^{q,s} = \frac{1}{r_q} \rbra{ \b_{s,q} - \frac{\b_{s,q}^T \e}{r_q + r_s} \e },
\end{equation}
where
\begin{equation} \label{eq: b}
\b_{q,s} = \rbra{ \bs \lambda_{C_q} \e^T + \e \bs \lambda_{C_s}^T
- \W_{C_q, C_s} - \E\sbra{\bs \lambda_{C_q} \e^T + \e \bs \lambda_{C_s}^T
- \W_{C_q, C_s}} } \e.
\end{equation}
Indeed, we must choose $\y = \y^{q,s}$ and $\z = \z^{q,s}$
to be solutions of the system
\newcommand{\I}{\bs I}
\begin{equation} \label{eq: yz system}
	\mat{{cc} r_s \I + \e\e^T & 0 \\ 0 & r_q \I + \e\e^T }
	\vect{ \y \\ \z } = \vect{\b^{q,s} \\ \b_{s,q} },
\end{equation}
to ensure that the complementary slackness conditions are satisfied.
Note that taking the inner product of each side
of \eqref{eq: yz system} with the vector \( (\e; -\e) \)
yields
\[
	(r_q + r_s)( \e^T \y -  \e^T \z) = \e^T \b^{q,s} - \e^T \b^{s,q}
	= 0,
\]
by the symmetry of \eq{\W}. This establishes that the solution
\((\y; \z)\) of \eqref{eq: yz system}
is also a solution of the (singular) system of equations,
\[
	\mat{{cc} r_s \I & \e\e^T  \\ \e\e^T & r_q \I }
	\vect{ \y \\ \z } = \vect{\b^{q,s} \\ \b_{s,q} },
\]
imposed by the complementary slackness conditions
${\S_{C_q, C_s} \e = \0}$ and ${\S_{C_s, C_q} = \0 }$.
Solving~\eqref{eq: yz system} for ${\y}$ and $\z$
using the Sherman-Morrison-Woodbury
Formula
yields the formula for $\y$ and $\z$ given by
\eqref{eq: y z}.
We set the remaining block $\bs\Xi_{C_{k+1}, C_{k+1} } = \0$.
\citet[Section~4.2]{ames2014guaranteed} provides further details.

Finally, we choose
\begin{equation} \label{eq: mu}
\tau = \min_{\substack{q,s=1,\dots, k \\ q \neq s }} \bra{\mu_{qq} - \mu_{qs}} \epsilon  \hat r =: \gamma\epsilon  \hat r,
\end{equation}
where \(\epsilon > 0\) is a parameter to be chosen later.
In particular, the analysis provided in Sections~\ref{sec: nonneg},~\ref{sec: S}, and~\ref{sec: wrap} establishes that
a suitable choice of $\epsilon$ exists if the hypothesis of
Theorem~\ref{thm: rec} is satisfied.

The entries of \(\S\) are chosen according to the stationarity condition~\eqref{E: w dual feas}, but we will also define
an auxiliary variable \(\St \in \Sigma^n\) as the following \((k+1)\times (k+1)\) block matrix:
\begin{equation} \label{e: tilde S def}
\St_{C_q, C_s} = \branchdef{
	\mu_{q,s}\e\e^T - \W_{C_q, C_s},  & \mbox{if } q, s \in \{1,\dots, k\}, \\
	\mu_{q,k+1} \e\e^T - \W_{C_q, C_{k+1}} + (\bs\lambda_{C_q} - \E[\bs\lambda_{C_q}] )\e^T ,  & \mbox{if } s = {k+1},  \\
	\mu_{k+1,s} \e\e^T - \W_{C_{k+1}, C_s} + \e (\bs \lambda_{C_s} - \E[\bs\lambda_{C_s}] )^T,   & \mbox{if } q = {k+1} .
}
\end{equation}
We next provide the following theorem, first stated by \citet[Theorem~4.2]{ames2014guaranteed},
which characterizes when the proposed dual variables satisfy the hypothesis
of Theorem~\ref{T:  w KKT conditions}.

\begin{theorem} \label{T: S opt conds}
	Suppose that the vertex sets $C_1, \dots, C_k$ define a $k$-disjoint-clique subgraph $K^*$ of
	the weighted complete graph $K_n = (V, \W)$, where
	$\W \in \Sigma^n$ is a random symmetric matrix sampled from the planted cluster model
  according to the distributions \(\{\Omega_{ij}\}\) with means \( \bra{\mu_{ij}} \)
  and variances $\{\sigma^2_{ij}\}$.
	Let $r_1, \dots, r_{k+1}$, and  $\hat r$ be defined as in Theorem~\ref{thm: rec}.
	Let $\X^*$ be the feasible solution for \eqref{eq: kdc sdp} corresponding to $C_1, \dots, C_k$ defined
	by \eqref{eq: cluster X}. Let $\tau \in \R,$ $\bs\lambda \in \R^n,$  and $\bs\Xi \in \R^{n\times n}$
	be chosen according to  \eqref{eq: mu}, \eqref{eq: lam q},  and \eqref{eq: Xi choice 1}, respectively,
	and let $\St$ be chosen according to \eqref{e: tilde S def}.
	Suppose that the entries of  $\bs\lambda$ and  $\bs\Xi$ are nonnegative.
	Then
	$\X^*$ is optimal for \eqref{eq: kdc sdp}, and $K^*$  is the maximum density
	$k$-disjoint-clique
	subgraph of $K_n$ corresponding to $\W$, if
	\begin{equation} \label{eq: St bound}
	\| \St\| \le  \epsilon \gamma \hat r.
	\end{equation}
	Moreover, if \eqref{eq: St bound} is satisfied and
	\begin{equation} \label{a: block weights}
	r_s \e^T \W_{C_q, C_q} \e > r_q \e^T \W_{C_q, C_s} \e,
	\end{equation}
	for all $q, s \in \{1,\dots,k\}$ such that $q\neq s$,
	then $\X^*$ is the unique optimal solution of \eqref{eq: kdc sdp} and
	$K^*$ is the unique maximum density $k$-disjoint-clique
	subgraph of $K_n$.
\end{theorem}

The proof of Theorem~\ref{T: S opt conds} is nearly identical to that by \citet[Theorem~4.2]{ames2014guaranteed}, and is
omitted.
Theorem~\ref{T: S opt conds} provides a clear roadmap for the remainder of the proof; if we can show that if \(\W\) is sampled from the planted cluster
model satisfying \eqref{eq:smallest_largest} then \(\bs \lambda\)
and \(\bs \Xi\) are nonnegative and \(\|\St\| \le \epsilon \gamma \hat r\) with high probability, then we will have established that we can recover the underlying block
structure with high probability in this case. We establish the necessary bounds on \(\bs \lambda,\) \(\bs \Xi,\) and \(\|\St\|\)
in the following sections.

\subsection{Nonnegativity of \(\bs\lambda\) and \(\bs\Xi\)}
\label{sec: nonneg}
We first establish that the entries of \(\bs \Xi\), as constructed according to \eqref{eq: Xi choice 1},
are nonnegative with high probability.
To do so, we will make repeated use of the following specialization of the
Bernstein inequality which provides a  bound on the tail of a sum of bounded independent random variables; see
\citet[Section~2.8]{lugosi2009}, for more details regarding  the Bernstein inequality.

\begin{theorem}
	\label{thm: Hoeffding}
	Let $x_1, \dots, x_m$ be \iid variables with mean $\mu$
	and variance $\sigma^2$.
	Let $S = x_1 + \cdots + x_m$. Then
	\begin{equation}
	\label{eq:hoeffding}
	Pr\rbra{ |S - \mu m| > 6 \max \bra{\sqrt{\sigma^2 m \log T}, \log T  }}
	 \le 2 T^{-6},
	\end{equation}
	for all $T > 1$.
\end{theorem}

The following bound
on the parametrizing vectors \(\y^{q,s}\) and \(\z^{q,s}\) in the choice of the \((C_q, C_s)\) block of \(\bs\Xi\) defined by~\eqref{eq: Xi choice 1}
is an immediate consequence of Theorem~\ref{thm: Hoeffding}.

\begin{lemma}
	\label{lem: yz bound}
	There exists constant $c > 0$ such that
	\begin{equation*}
		\|\bs{y}^{q,s}\|_{\infty} +  \|\bs{z}^{q,s}\|_{\infty} \leq
		c \max \bra{ \sqrt{ \frac{\tilde \sigma^2 \log n}{\hat r}} ,
			\frac{\log n}{\hat r} },
	\end{equation*}
	w.h.p.,
	where $\tilde \sigma := \max \{ \sigma_{ij} : i,j = 1,2,\dots, k+1 \}$,
	for all \(q,s \in \{1,\dots, k+1\}\) such that 	$q \neq s$.
\end{lemma}

For \(q,s \in \{1,\dots, k+1\}\) such that \(q \neq s\), we define \(\y^{q,s}\) and \(\z^{q,s}\) as in \eqref{eq: y z}. To bound
the absolute values of the entries of \(\y^{q,s}\) and \(\z^{q,s}\), we must estimate the sums \(\e^T \W_{C_q, C_q} \e\),
\(\e^T \W_{C_s, C_s} \e \) and \(\e^T \W_{C_q, C_s} \e \); applying Theorem~\ref{thm: Hoeffding} to bound the tails
of these sums yields Lemma~\ref{lem: yz bound}. See Appendix~\ref{app: yz} for the full argument.

We have the following bound on the entries of \(\bs \Xi\) as an immediate consequence of  Lemma~\ref{lem: yz bound}.

\begin{proposition}
	\label{thm: Xi bound}
	Suppose that $\bra{\mu_{ij}}$ satisfy \eqref{eq:smallest_largest}.
	Then there exists constant $c>0$ such that each entry of $\bs{\Xi}$ is nonnegative w.h.p.~if \(\epsilon\) 	satisfies
	\begin{equation} \label{eq: xi eps}
	0 < \epsilon \leq
		1 - \frac{c}{\gamma}\max \bra{ \sqrt{ \frac{\tilde \sigma^2 \log n}{\hat r}} ,
			\frac{\log n}{\hat r}  }.
	\end{equation}
\end{proposition}

\begin{proof}
	Fix $q,s \in \{1, \dots, k\}$ such that $q \neq s$.
	By construction, we have
	\begin{align*}
	\bs{\Xi}_{C_q,C_s} 
	&= \E\sbra{\bs{\lambda}_{C_q}\bs{e}^T+\bs{e}\bs{\lambda}_{C_s}^T - \bs{W}_{C_q,C_s}}+ \bs{y}^{q,s}\bs{e}^T+\bs{e}
	\left(\bs{z}^{q,s}\right)^T \\
	&= \left(\frac{1}{2} \left(\mu_{qq} - \frac{\tau}{r_q} \right) + \frac{1}{2} \left(\mu_{ss} - \frac{\tau}{r_s} \right) - \mu_{qs} \right)
	\bs{e}\bs{e}^T + \bs{y}^{q,s}\bs{e}^T+\bs{e}\left(\bs{z}^{q,s}\right)^T.
	\end{align*}
	Using
	\eqref{eq: mu} and
	 Lemma~\ref{lem: yz bound}, we see that
	\begin{align*}
	\Xi_{ij} &\geq \frac{1}{2} \left(\mu_{qq} - \gamma \epsilon \right) + \frac{1}{2} \left(\mu_{ss} - \gamma \epsilon \right) - \mu_{sq} -
	\|\bs{y}^{q,s}\|_{\infty}-\|\bs{z}^{q,s}\|_{\infty} \\
	& \ge (1 - \epsilon)\gamma - c \max \bra{ \sqrt{ \frac{\tilde \sigma^2 \log n}{\hat r}} ,
			\frac{\log n}{\hat r} },
	\end{align*}
	for all $i \in C_q$, $j\in C_s$ w.h.p., where $c$ is the constant appearing in
	Lemma~\ref{lem: yz bound}. Note that
	the right-hand side of this inequality is nonnegative	if and only if
	\[\epsilon \leq 1 - \frac{c}{\gamma}\max \bra{ \sqrt{ \frac{\tilde \sigma^2 \log n}{\hat r}} ,
			\frac{\log n}{\hat r} } . \]
	The argument for the case when one of $q$ or $s$ is equal to $k+1$
	follows analogously.
	Applying the union bound over all blocks of \(\bs \Xi\) shows that each entry
	of $\bs{\Xi}$ is nonnegative
	w.h.p.~if $\epsilon$ satisfies \eqref{eq: xi eps}.
\end{proof}
\medskip

We have an analogous result ensuring that the entries of \(\bs\lambda\) are nonnegative
with high probability; we present the proof of this result in Appendix~\ref{app: lam}.

\begin{proposition} \label{prop: lam}
	Suppose $\bra{\mu_{ij}}$ satisfy \eqref{eq:smallest_largest}. Then there
	exists constant $c' > 0$ such that each entry of $\bs{\lambda}$ is nonnegative w.h.p.~if $\epsilon$ satisfies
	\begin{equation} \label{eq: lam eps}
	0 < \epsilon \leq  \frac{1}{2\gamma}
	\rbra{ \mu_{qq}  - c' \max \bra{ \sqrt{  \frac{\sigma_{qq}^2 \log 		n }{r_q} },
			\frac{\log n }{r_q} } },
	\end{equation}
	for all $ q \in \{1,\dots, k\}$.
\end{proposition}

We conclude this section with a result ensuring that the uniqueness condition~\eqref{a: block weights} of
Theorem~\ref{T: S opt conds}
is satisfied for all \(q, s\in \{1,\dots, k\}\) such that \(q \neq s\);
we provide a proof in Appendix~\ref{A_unique}.

\begin{proposition}
\label{prop:unique}
	Suppose that
	\begin{equation} \label{eq: unique bound}
		\gamma \ge
		12 \max \bra{ \sqrt{ \frac{\tilde \sigma^2 \log n}{\hat r^2}},
		\frac{\log n}{\hat r^2} }.
	\end{equation}
	Then \(r_s \e^T \W_{C_q, C_q} \e > r_q \e^T \W_{C_q, C_s} \e \)
	for all \(q,s \in \{1,\dots, k\}\) such that \(q \neq s\) with high probability.
\end{proposition}

\subsection{A Bound on \(\bs{\tilde S}\)}
\label{sec: S}
It remains to establish the following bound on the spectral norm of the matrix \(\St\).

\begin{proposition} \label{prop: S bound}
	There exists scalars $C, C' > 0$ such that
	\begin{equation} \label{eq: S bound}
	\|\St\| \leq C   \max  \bra {\tilde \sigma\sqrt{n},  \sqrt{\log n}}
	+ C' \rbra{ \max \bra{\hat\sigma^2, \frac{\log n}{\hat r} } k r_{k+1} }^{1/2}
	+ \mu_{k+1, k+1} r_{k+1},
\end{equation}
	where $\hat \sigma^2 = \max_{q=1,\dots,k}\bra{\sigma^{2}_{qq}}$,
	with high probability.
\end{proposition}

The proof of Proposition~\ref{prop: S bound} follows the same structure as that of  \citet[Lemma~4.5]{ames2014guaranteed}.
In particular, we decompose \(\St\) as \(\St = \St_1 + \St_2 + \St_3,\) where
\begin{align}
\St_1 & = \E[ \W ] - \W,  \label{eq: St1 def}\\
[\St_2]_{C_q, C_s} & =
\begin{cases}  \left(\bs{\lambda}_{C_q}-\E\sbra{\bs{\lambda}_{C_q}}\right)\bs{e}^T,
& \mbox{if } s = k+1, \\
\bs{e}\left(\bs{\lambda}_{C_s}-\E\sbra{\bs{\lambda}_{C_s}}\right)^T, & \mbox{if } q = k+1,  \\
\bs{0}, & \mbox{otherwise, }\end{cases} \label{eq: St2 def} \\
[\St_3]_{C_q, C_s} &= 	\begin{cases} -\mu_{k+1, k+1} \bs{e}\bs{e}^T, & \mbox{if } q=s=k+1, \\
\bs{0}, & \mbox{otherwise.} \end{cases}
\end{align}

Note that $\|\St_3\| = \mu_{k+1, k+1}\|\bs{e}\bs{e}^T\|=\mu_{k+1, k+1} r_{k+1}$.
The following lemmas provide the necessary bounds on $\|\St_1\|$ and $\|\St_2\|$.

\begin{lemma} \label{lem: S1}
	Suppose that \(\St_1\) is constructed according to \eqref{eq: St1 def}
	for some \(\W\in \Sigma^n\) sampled from the heterogeneous planted cluster model.
	Then there exists constant $C > 0$ such that
	\begin{equation} \label{eq: S1 bound}
	\|\St_1\| \le C   \max  \bra {\tilde \sigma\sqrt{n},  \sqrt{\log n}},
	\end{equation}
	with high probability.
\end{lemma}

\begin{lemma} \label{lem: S2}
	Suppose that \(\St_2\) is constructed according to \eqref{eq: St2 def}
	for some \(\W\in \Sigma^n\) sampled from the heterogeneous planted cluster model.
	Then there exists constant $C'> 0$ such that
	\begin{equation*} 
		\|\St_2 \| \le C' \rbra{ \max \bra{\hat\sigma^2, \frac{\log n}{\hat r} } k r_{k+1} }^{1/2},
	\end{equation*}
	with high probability, where $\hat \sigma := \max_{q=1,\dots, k} \sigma_{qq}$.
\end{lemma}

We delay the proof of Lemmas~\ref{lem: S1} and~\ref{lem: S2} until Appendix~\ref{A_S1_bound}
and Appendix~\ref{A_S2_bound}, respectively.
Combining the three bounds on \(\|\St_1\|,\) \(\|\St_2\|\), and \(\|\St_3\|\) and
applying the triangle inequality one last time shows that~\eqref{eq: S bound} holds
with high probability.

\subsection{The Conclusion of the Proof}
\label{sec: wrap}
According to Theorem~\ref{T: S opt conds}, it suffices to prove that
$\|\St\| \leq \epsilon \gamma \hat{r}$ is satisfied with high probability in order to prove Theorem~\ref{thm: rec}.
According to Proposition~\ref{prop: S bound}, if
\begin{equation} \label{eq:eps_req1}
\resizebox{0.92\textwidth}{!}{$
	\gamma\epsilon \hat r \geq  C \max\left\lbrace \tilde \sigma \sqrt{n},  \sqrt{\log n} \right\rbrace
+ C' \rbra{ \max \bra{\hat\sigma^2, \frac{\log n}{\hat r} } k r_{k+1} }^{1/2} + \mu_{k+1, k+1} r_{k+1},
	$}
\end{equation}
then $\|\St\| \leq \epsilon \gamma \hat{r}$ holds with high probability.
Hence, we have three conditions, \eqref{eq: xi eps}, \eqref{eq: lam eps} and \eqref{eq:eps_req1}, on $\epsilon > 0$ that need to be satisfied simultaneously;
choosing any \(\epsilon > 0\) satisfying all three establishes the desired recovery guarantee.
We see that   \eqref{eq: xi eps} and \eqref{eq:eps_req1} can be simultaneously fulfilled if
\[ \resizebox{\textwidth}{!}{\(
	1 - \frac{c}{\gamma}\max \bra{ \sqrt{ \frac{\tilde \sigma^2 \log n}{\hat r}} ,
			\frac{\log n}{\hat r}}
			\geq
			\frac{1}{\gamma \hat{r}}\left( C \max\left\lbrace \tilde \sigma \sqrt{n},  \sqrt{\log n} \right\rbrace
		+ C' \rbra{ \max \bra{\hat\sigma^2, \frac{\log n}{\hat r} } k r_{k+1} }^{1/2} + \mu_{k+1, k+1} r_{k+1}\right) \)},
\]
which holds if and only if
\begin{equation} \label{eq:general_low_bound_r}
\resizebox{0.91\textwidth}{!}{\(
	 \hat r \rbra{\gamma - c \max \bra{ \sqrt{ \frac{\tilde \sigma^2 \log n}{\hat r}} , \frac{\log n}{\hat r} }  }
			\geq   C \max\left\lbrace \tilde \sigma \sqrt{n},  \sqrt{\log n} \right\rbrace
		+ C' \rbra{ \max \bra{\hat\sigma^2, \frac{\log n}{\hat r} } k r_{k+1} }^{1/2} + \mu_{k+1, k+1} r_{k+1}.
\)}
\end{equation}
Next, we see that \eqref{eq:eps_req1} and \eqref{eq: lam eps} are simultaneously fulfilled if
\[
\resizebox{\textwidth}{!}{\(
		\frac{1}{2\gamma} \rbra{ \mu_{qq}  - c' \max \bra{ \sqrt{  \frac{\sigma_{qq}^2 \log n }{\hat r} }, \frac{\log n }{\hat r} }  }
			\geq
			\frac{1}{\gamma \hat{r}} \rbra{  C \max\left\lbrace \tilde \sigma \sqrt{n},  \sqrt{\log n} \right\rbrace
		+ C' \rbra{ \max \bra{\hat\sigma^2, \frac{\log n}{\hat r} } k r_{k+1} }^{1/2} + \mu_{k+1, k+1} r_{k+1} }
\)},
\]
which holds if and only if
\begin{equation} \label{eq:general:low_bound_r2}
\resizebox{0.91\textwidth}{!}{\(
	\hat r \rbra{ \mu_{qq}  - c' \max \bra{ \sqrt{  \frac{\sigma_{qq}^2 \log n }{r_{qq}} }, \frac{\log n }{r_{qq}} }  }
			\geq
			 2\rbra{  C \max\left\lbrace \tilde \sigma \sqrt{n},  \sqrt{\log n} \right\rbrace
	 + C' \rbra{ \max \bra{\hat\sigma^2, \frac{\log n}{\hat r} } k r_{k+1} }^{1/2} + \mu_{k+1, k+1} r_{k+1} }
\)}.
\end{equation}
Finally, suppose that we choose the parameter $c_4 > \max\{c, c', 12\}$ so that gap condition~\eqref{eq: unique bound}} is satisfied  and
\[
	\gamma > \max\{c, c', 12\} \max \bra{ \sqrt{  \frac{\tilde \sigma^2 \log n }{\hat r} }, \frac{\log n }{\hat r} }.
\]
Then there exist constants $c_1, c_2, c_3$, depending on $c_4$, such that
\eqref{eq:general_low_bound_r} and \eqref{eq:general:low_bound_r2} are satisfied, i.e., there exists
$\epsilon$ satisfying   \eqref{eq: xi eps}, \eqref{eq: lam eps} and
\eqref{eq:eps_req1}  simultaneously, if
\begin{align*}
	\gamma \hat r  \ge  c_1  &\max\left\lbrace \tilde \sigma\sqrt{n},  \sqrt{\log n} \right\rbrace
			+ c_2 \rbra{ \max \bra{\sigma_1^2, \frac{\log n}{\hat r} } k r_{k+1} }^{1/2} + c_3 \mu_{k+1, k+1} r_{k+1} .
\end{align*}
This concludes the proof of Theorem~\ref{thm: rec}. 

\section{Numerical Methods and Simulations}
\label{sec: num}
We conclude with a discussion of an algorithm for solution of~\eqref{eq: kdc sdp} based on the alternating direction method of multipliers (ADMM),  and provide the results of a series of experiments that empirically verify the phase transitions predicted in Section~\ref{sec: rec}.
In particular, we randomly sample graphs \(G = (V,\W)\) from the planted cluster model and compare the
optimal solution of \eqref{eq: kdc sdp} with the planted partition.

\subsection{Alternating Direction Method of Multipliers for the Densest $k$-Disjoint Clique Problem}
We solve \eqref{eq: kdc sdp} iteratively using the algorithm proposed by \cite{ames2014guaranteed}.
Specifically, we split the decision variable \eq{\X} to obtain the equivalent formulation
\begin{equation*} 
\max \bra{ \tr(\W\Y) : \X - \Y = \0, \X\e \le \e, \X \ge \0, \; \tr \Y = k, \Y \in \Sigma^V_+ }.
\end{equation*}
We then apply an approximate dual ascent scheme to maximize the augmented Lagrangian
\[
L_\rho(\X,\Y,\Z) = \tr(\W\Y) - \tr(\Z(\X-\Y) ) + \frac{\rho}{2} \| \X - \Y\|^2_F,
\]
where \(\rho > 0\) is a penalty parameter for violation of the linear equality constraint  \(\X-\Y = \0\).
In particular, we minimize \(L_\rho\) with respect to \(\Y\) and \(\X\) successively, and then update
\(\Z = \Z - \rho(\X-\Y)\) using approximate gradient ascent.

We update $\Y$ as the minimizer of the subproblem
\begin{equation*} 
	\Y^{t+1} = \argmin_{\Y \in \Sigma^n_+} \bra{  \left\| \Y - \rbra{ \X^t - \frac{\W + \Z^t}{\rho}}\right\|^2_F : \tr \Y = k},
\end{equation*}
where $(\X^t, \Y^t, \Z^t)$ is the current iterate after $t$ iterations.
That is, $\Y^{t+1}$ is the projection of the matrix $\bs{U}^t := \X^t - (\W + \Z^t)/\rho$ onto the intersection of the positive semidefinite cone and the set of matrices with trace equal to zero.
Such a projection can be computed explicitly by projecting the vector of eigenvalues $\bs{\lambda}^t$ of $\bs{U}^t$ onto the nonnegative simplex $\{\y \in \R^n: \e^T \y = k, \; \y \ge \bs 0\}$.
\citet[Proposition~2.6] {zhang2011penalty} and \citet{van2008probing} can be consulted for further details.

We update $\X^{t+1}$ as the optimal solution of
\begin{equation}\label{eq:X-subprob}
	\X^{t+1} = \argmin_{\X \in \R^{n\times n}} \bra{ \left\|\X - \rbra{\Y^t + \Z^t/\rho} \right\|^2_F: \X \ge \0, \; \X \e \le \e }.
\end{equation}
Applying strong duality, we know that the optimal solution of~\eqref{eq:X-subprob} is given by
\begin{equation*} 
	\X^{t+1} = \sbra{ \rbra{\Y^{t+1} + \Z^t/\rho} - \frac{\z^* \e + \e (\z^*)^T}{2}}_+,
\end{equation*}
where the operator $[\cdot]_+$ is the projection onto the symmetric nonnegative cone $\Sigma^V \cap \R^{V\times V}_+$ given by $[[\Z]_+]_{ij} = \max \{0, Z_{ij}\}$ for all $\Z \in \Sigma^V$, and $\z^*$ is the optimal solution of the dual problem of~\eqref{eq:X-subprob} given by
\begin{equation}\label{eq:X-dual}
	\min_{\z \ge \0} \frac{1}{2} \left\|\sbra{ \rbra{\Y^{t+1} + \Z^t/\rho} - \frac{\z \e + \e \z^T}{2}}_+ \right\|_F^2 + \z^T \e - \frac{1}{2} \|\Y^{t+1} + \Z^t/\rho \|^2_F.
\end{equation}
The objective function of the dual problem~\eqref{eq:X-dual} is differentiable and coercive in $\z$, so it can be solved efficiently by applying the spectral projected gradient method of~\cite{birgin2000nonmonotone}.
We complete each iteration by performing an approximate dual ascent step to update the dual variable $\Z^{t+1}$.
We stop the projected gradient method when the relative duality gap, given by $|v_p^{(t)} - v_d^{(t)}|/\max\{v_p^{(t)}, 1\}$, and primal constraint violation are both smaller than a desired error tolerance.
We summarize the algorithm as Algorithm~\ref{alg:ADMM}.
Please see the work of \citet[Section 6]{ames2014guaranteed} for further implementation details.

\newcommand{\V}{\bs{V}}
\begin{algorithm}[h!]
	\caption{ADMM for~\eqref{eq: cluster SDP}}
	\label{alg:ADMM}
	\begin{algorithmic}

	\algrenewcommand\algorithmicfor{For}
	\algrenewcommand\algorithmicdo{}
	\algrenewcommand\algorithmicend{End}

	\State{\textbf{Input:} Initial iterates $\X^0 = \Y^0 = \Z^0 = \0$, augmented Lagrangian parameter $\rho > 0$, and stopping tolerance $\epsilon > 0$.}

	\State{\textbf{Output:} Approximate solution $(\X^*, \Y^*, \Z^*)$ of \eqref{eq: cluster SDP}.}

	\For{$t = 0, 1,2 \dots $ until converged}
	\State
		Compute spectral decomposition $\V^t \Diag \bs{\lambda}^t (\V^t)^T = \bs U^t = \X^t - (\W + \Z^t)/\rho$.
	\State
		Project  $\bs \lambda^t$ onto the nonnegative simplex $\{\y \in \R^n: \e^T \y = k,\; \y \ge \0\}$ to obtain $\bs{\bar\lambda}^t$.
	\State
		Update $\Y^{t+1} = \V^t \Diag \bs{\bar\lambda}^t (\V^t)^T$.
	\parState{Compute approximate optimal solution $\z^*$ of the dual subproblem~\eqref{eq:X-dual} using spectral projected gradient method of~\cite{birgin2000nonmonotone}.}
	\State
		Update $\X^{t+1}= \sbra{ \rbra{\Y^{ts+1} + \Z^t/\rho} - \frac{\z^* \e + \e (\z^*)^T}{2}}_+.$
	\State
		Update $\Z^{t+1}$ using approximate dual ascent
		\[
			\Z^{t+1} = \Z^t - \rho(\X^{t+1} - \Y^{t+1}).
		\]
	\State
		Compute primal feasibilty gap
		\[
			pfeas = \min\bra{\min_{ij} Y^t_{ij}, \min \rbra{\e - \Y^{t}\e} }.
		\]
	\parState{Compute estimates of primal and dual objective values
    		(note that $v_d^{(t+1)}$ is not necessarily a lower bound on the optimal dual value, but is asymptotically converging to the optimal dual value):
		\[
			v_p^{(t+1)} = \tr(\W\Y^{t}) \hspace{0.25in}
			v_d^{(t+1)} = k \lambda_{min}(\W + \Z^{t+1}) - \tr(\X^{t+1}\Z^{t+1}).
		\]    }
	\State
		Calculate relative duality gap
		\[
			relgap = \frac{|v_p^{(t+1)} - v_d^{(t+1)} |}{\max\bra{|v_p^{(t+1)}|, 1} }.
		\]
	\State
		Declare sequence of iterates to have converged if $relgap < \epsilon$ and $pfeas >- \epsilon$.
	\EndFor
	\end{algorithmic}
\end{algorithm}

\subsection{Empirical Verification of Exact Recovery}

We perform two sets of experiments, one to illustrate the recovery guarantee for dense graphs sampled from the heterogeneous planted cluster model and another to illustrate the guarantee
when the noise is sparse.
For the dense graph experiments, we fix \(n = 1000\), and sample \(10\) graphs
from the heterogeneous planted cluster model corresponding to the Bernoulli
distributions \(\Omega_{ij} = Bern(p_{ij})\)  with probabilities of success $p_{ij}$ given by
\[
	p_{ij} := \begin{cases} \rbra{ 1 - \rbra{\frac{0.35}{k+1}}i } p, &\text{if } i = j, \\
								\rbra{ 1 - \rbra{\frac{0.35}{k+1}} \min\{i,j\} } q, &\text{if } i \neq j,
							\end{cases}
\]
for $q = 0.25$ and each $p=\{0.25, 0.275, 0.3, \dots, 0.975, 1\}$ and
\(\hat r \in \{20, 40, \dots, 500\}\).
We choose the number of clusters $k=\lfloor n/\hat r \rfloor$
and distribute the remaining $n - k\hat r$ nodes
evenly among $k-1$ clusters to ensure that at least one cluster
has minimum size.
Under this choice of $p_{ij}$ the smallest gap between the in-cluster and between-cluster means occurs when $i=1$ and $j = k$; this implies that
\begin{equation*} 
	\gamma = \rbra{ 1 - \frac{0.35 k }{k+1}}p - q.
\end{equation*}
For each graph \(G\), we call the ADMM algorithm sketched above to solve \eqref{eq: kdc sdp};
in the algorithm, we use penalty parameter \(\rho = \min \bra{ \max \bra{ 5n/k, 80}, 500}/2 \),
stopping tolerance \(\epsilon = 10^{-4}\), and maximum number of iterations \(100\).
We declare the block structure of \(G\) to be recovered if \(\|\X^* - \X_0\|^2_F / \|\X_0\|^2_F < 10^{-3} \),
where \(\X^* \) is the solution returned by the ADMM algorithm and \(\X_0\) is the proposed solution
given by \eqref{eq: cluster X}.
Note that Theorem~\ref{thm: rec} implies that we should expect exact recovery (w.h.p.) provided that
$
	\gamma \hat r = \Omega\rbra{\sqrt{\tilde \sigma^2n}}.
$
Figure~\ref{fig: rec}(a) illustrates the empirical success rate for each choice of \(\hat r\) and \(p\), as well as the curve \( p =\big((k+1)/(0.65k + 1)\big)(q + \half {\sqrt{n}}),\) where we use the upper bound $\tilde \sigma^2 \le 1/4$ to estimate the constant term in~\eqref{eq:smallest_largest}.

\begin{figure}[t]
	\centering
	\subfloat[{Dense Noise}]{\includegraphics[width=0.48\textwidth]{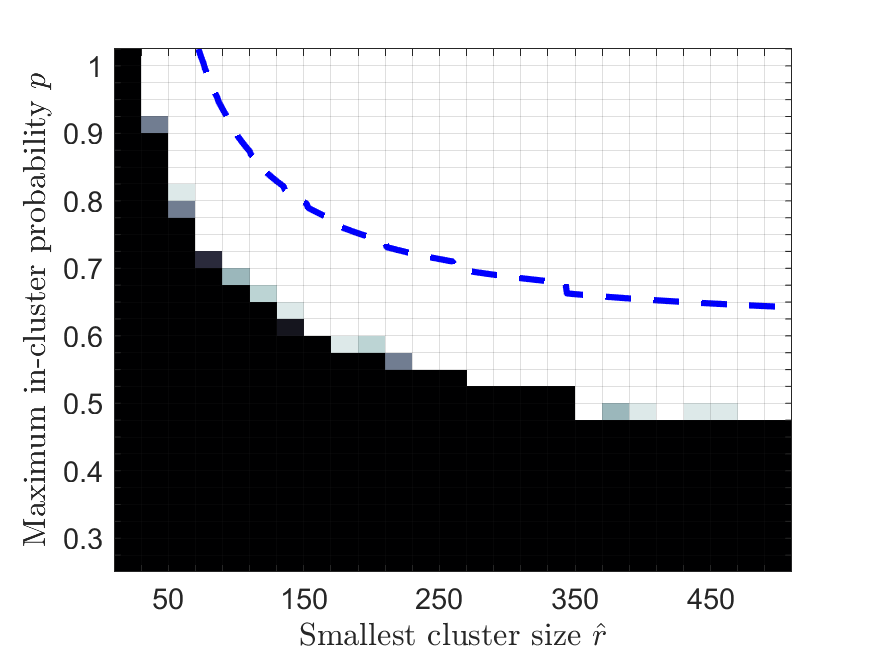} }
	\subfloat[{Sparse Noise}]{\includegraphics[width=0.48\textwidth]{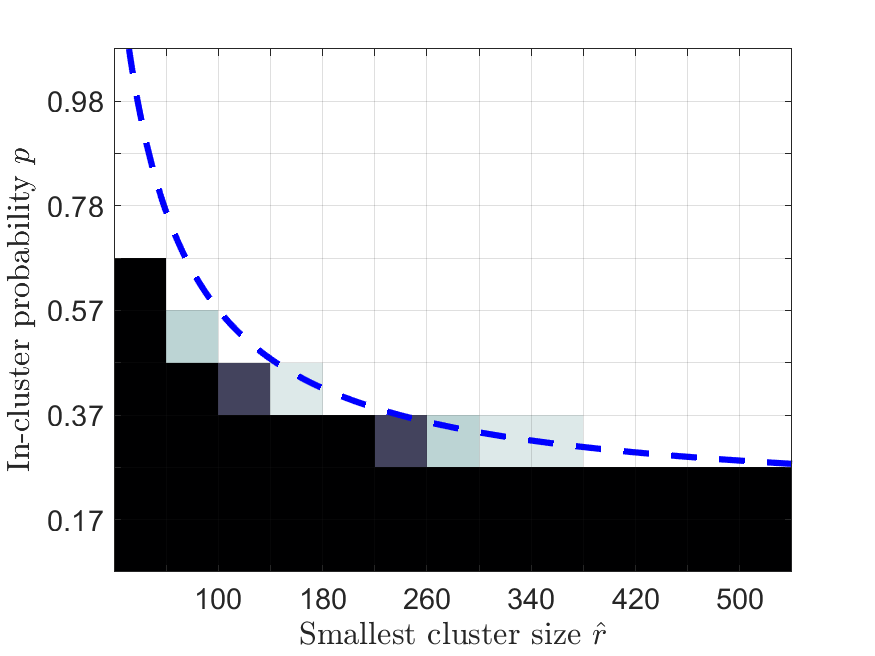}}
		\caption{Empirical recovery rate for $n$-node graph with $k$ planted cliques of size at least $\hat r$ and \(\W\) generated
		according to the planted heterogeneous cluster model with distributions $\Omega_{ij} = Bern(p_{ij})$. Brighter colors indicate higher
		rates of recovery, with black corresponding to 0 recoveries and white corresponding to 10 recoveries (out of 10 trials).
		The dashed curves indicate the phase transition to perfect recovery predicted by Theorem~\ref{thm: rec}.}
	\label{fig: rec}
\end{figure}

We perform identical experiments for graphs sampled from the homogeneous planted cluster model with sparse noise.
In particular, we fix \( n =1000\) and set $q = 1/\sqrt{n}$. We then sample \(10\) graphs
from the planted cluster model corresponding to the Bernoulli
distributions \(\Omega_{ij} = Bern(p)\) if $i =j$ and \(\Omega_{ij} = Bern(q)\) if $i\neq j$ for each
\( \hat r \in \{20,60,\dots, 440, 500\}\) and $p = tq$ for $10$ equally spaced scaling factors $t$ between $2$ and $\floor{\sqrt{n}}$.
As before, we set $k = \floor{n/\hat r}$ and distribute the remaining nodes equally amongst the clusters so that the smallest has size $\hat{r}$ and $r_{k+1} = 0$.
For each graph \(G\), we call the ADMM algorithm to solve \eqref{eq: kdc sdp} (with the same parameters as before)
and declare the block structure of \(G\) recovered if \(\|\X^* - \X_0\|^2_F / \|\X_0\|^2_F < 10^{-3} \).
Theorem~\ref{thm: rec} suggests that we should expect recovery of the cluster structure in the case that
\[
	p >  \Omega\rbra{\frac{1}{\sqrt{n}} + \frac{n^{1/4}}{\hat r}},
\]
for this particular choice of $p$ and $q$.
Note that this implies that we have perfect recovery (w.h.p.) for $\hat r = \Omega(n^{1/4})$, rather than $\Omega(\sqrt{n})$ (as observed in the dense case).
Figure~\ref{fig: rec}(b) provides the empirical success rate for each choice of \(\hat r\) and \(p\), as well as the curve \(p = 1/\sqrt{n} + n^{1/4}/\hat{r}\).
It is clear that we are able to recover significantly smaller clusters under sparse noise than under dense noise, in accordance with \eqref{eq:smallest_largest}.

\section{Conclusions}
\label{sec: con}
We have established theoretical guarantees for graph clustering via a semidefinite relaxation of the densest \(k\)-disjoint problem.
These results add to the growing corpus of evidence that clustering, while intractable in general,
is possible if we seek to cluster clusterable data, i.e., data consisting of well-defined and well-separated groups of similar items.
Moreover, our results provide further evidence that the \(\omega(\sqrt{n})\) barrier can be broken for perfect cluster recovery in approximately sparse graphs and, specifically, that the size of recoverable clusters
scales logarithmically with \(n\) at worst in the special case that all clusters
are roughly the same size.
Finally, our semidefinite relaxation requires only an estimate of the
number of clusters present in the data as input.

Our results suggest several areas of further research.
The numerical simulations suggest that our theoretical guarantees may be overly conservative, especially in the dense noise case;
further investigation is needed to determine if tighter estimates on the minimum size of clusters efficiently recoverable
exist.
Moreover, our model assumes clusters are disjoint. This is clearly not met in many practical applications; for example,
returning to the social networking realm,
users may belong to several overlapping communities.
It would be worthwhile to see how our model and recovery guarantees can be modified to address overlapping clusters.
Finally, our algorithm for graph clustering requires the solution of a semidefinite program, which may be impractical for
even moderately large graphs.
For example, the proposed algorithm, based on the ADMM, has per-iteration cost of $O(n^3)$ flops per iteration, primarily to compute the spectral decomposition needed to update $\Y$. Classical methods based on interior-point methods will scale even more poorly. Efficient, scalable methods for solving this semidefinite relaxation, and semidefinite programming in general, are needed.

\subsubsection*{Acknowledgements}
{We are grateful to John Bruer and Joel Tropp for their insights and helpful suggestions.
Aleksis Pirinen was supported by a California Institute of Technology Summer Undergraduate Research Fellowship (SURF) using funds provided by Office of Naval Research (ONR) award N000014-11-1002. Brendan Ames was supported by University of Alabama Research Grants RG14678 and RG14838.}
%
\appendix

\section{Proof of Lemma~\ref{lem: yz bound}}
\label{app: yz}
We give the full proof of Lemma~\ref{lem: yz bound} in this appendix.
\begin{proof}
We fix $q,s \in \{1,\dots,k\}$ such that $q \neq s$ and assume without loss of generality that $r_q \leq r_s$.
By the definition \eqref{eq: y z} of $\bs{y} := \bs{y}^{q,s}$ and the triangle inequality, we have
\begin{equation*}
	\|\bs{y}\|_{\infty} \leq \frac{1}{r_s}\left( \|\bs{b}_{q,s}\|_{\infty} + \frac{|\bs{b}_{q,s}^T \bs{e}|}{r_q +r_s}\right).
\end{equation*}
For simplicity, let \(\b_1:= \b_{q,s}\) and \(\b_2 := \b_{s,q}\).
It follows from \eqref{eq: b} and our choice of $\bs{\lambda}$  that the $i$th element of $\bs{b}_1$, denoted $b_i^1$, is given by
$$
	b_i^1 = r_s \left( \lambda_i -\frac{1}{2r_q}(\mu_{qq} r_q - \tau) \right) + \left(\bs{\lambda}_{C_s}^T\bs{e} - \frac{1}{2}(\mu_{ss} r_s - \tau) \right) - \left( \sum_{j \in C_s} w_{ij} - \mu_{qs} r_s \right).
$$
It follows from the definition \eqref{eq: lam q} of $\bs{\lambda}_{C_s}$  that
$$
	\bs{\lambda}_{C_s}^T\bs{e} = \frac{1}{2 r_s}\left(\bs{e}^T\bs{W}_{C_s,C_s}\bs{e} - r_s \tau \right),
$$
which implies that
$$
	\left| \bs{\lambda}_{C_s}^T\bs{e} - \frac{1}{2}(\mu_{ss} r_s - \tau)  \right| = \frac{1}{2r_s}\left|\bs{e}^T\bs{W}_{C_s,C_s}\bs{e}-\mu_{ss} r_s^2 \right|.
$$
Applying \eqref{eq:hoeffding} with $T = n$ to the right-hand side in the equation above shows that
\begin{equation} \label{cluster sum}
	|\e^T \W_{C_s, C_s} \e - \mu_{ss} r_s^2 |
	 \le 6 \max \{ \sqrt{\sigma_{ss}^2 r_s^2 \log n}, \log n \},
\end{equation}
with high probability,
which in turn implies that
\begin{equation*} 
	\left| \bs{\lambda}_{C_s}^T\bs{e} - \frac{1}{2}(\mu_{ss} r_s - \tau)
	\right| \le 3 \max\bra{ \sqrt{\sigma_{ss}^2  \log n}, \frac{\log n}{r_s} },
\end{equation*}
with high probability.
Similarly, applying \eqref{eq:hoeffding} with $T=n$ to the sum $\sum_{j\in C_s} w_{ij}$ shows that
\begin{equation*} 
	\abs{ \sum_{j\in C_s} w_{ij} - \mu_{qs} r_s }
		\le 6 \max \bra{ \sqrt{\sigma^2_{qs} r_s \log n}, \log n },
\end{equation*}
for all $i \in C_q$, with high probability.
Finally, we note that
\begin{equation*}
\left|\lambda_i - \frac{1}{2r_q}(\mu_{qq} r_q - \tau)\right| \leq \frac{1}{r_q}\left|\sum_{j \in C_q} w_{ij} - \mu_{qq} r_q \right| + \frac{1}{2r_q^2}\left|\bs{e}^T \bs{W}_{C_q,C_q} \bs{e} - \mu_{qq} r_q^2\right|.
\end{equation*}
We bound the first term in the sum using \eqref{eq:hoeffding} with $T=n$,
which establishes that
\begin{align*}
 \left|\sum_{j \in C_q} w_{ij}- \mu_{qq} r_q\right|
 	\le 6 \max \bra{ \sqrt{\sigma^2_{qq} r_q \log n}, \log n },
\end{align*}
w.h.p.,~and note that the second term has upper bound
\begin{align*}
 \left| \bs{e}^T \bs{W}_{C_q,C_q} \bs{e} - \mu_{qq} r_q^2 \right|
\le 6 \max \bra{\sqrt{\sigma_{qq}^2 r_q^2 \log n}, \log n },
\end{align*}
w.h.p.,~by a calculation identical to that used to obtain~\eqref{cluster sum}.
Applying these bounds using the triangle inequality and the
union bound over all $i \in C_q$, we conclude that
\begin{align}
&\|\bs{b}_1\|_{\infty} \leq r_s \left|\lambda_i -\frac{1}{2r_q}(\mu_{qq} r_q - \tau)\right|
+ \left|\bs{\lambda}_{C_s}^T\bs{e} - \frac{1}{2}(\mu_{ss} r_s - \tau)\right| + \left|\sum_{j \in C_s} w_{ij} - \mu_{qs} r_s\right| \notag\\
 & \leq  \;r_s \left(\frac{6}{r_q}\max \bra{ \sqrt{\sigma^2_{qq} r_q \log n}, \log n } + 3 \max\bra{ \sqrt{\sigma_{qq}^2  \log n}, \frac{\log n}{r_q}} \right) \notag \\
& \hspace{0.55in} + 3  \max\bra{ \sqrt{\sigma_{ss}^2  \log n}, \frac{\log n}{r_s}}
+ 6 \max\bra{ \sqrt{\sigma_{qs}^2  r_s \log n}, \log n}  \notag\\
 & =  O \rbra{ r_s \max \bra{  \sqrt{ \frac{ \tilde{\sigma}^2 \log n}{\hat{r}} }, \frac{ \log n}{\hat{r}} } }, \label{eq:b1_inf_orig}
\end{align}
with high probability.

We next bound $\left|\bs{b}_1^T \bs{e}\right|$. We have
$$ \bs{b}_1^T \bs{e} = r_s \left(\bs{\lambda}_{C_q}^T \bs{e} - \frac{1}{2}\left(\mu_{qq} r_q - \tau\right)\right) + r_q \left(\bs{\lambda}_{C_s}^T \bs{e} - \frac{1}{2}\left(\mu_{ss} r_s - \tau\right)\right) + \left(\mu_{qs} r_s r_q - \bs{e}^T \bs{W}_{C_q,C_s}\bs{e}\right).$$

Applying \eqref{eq:hoeffding} to bound the sum of the entries of
$\bs{W}_{C_q,C_s}$ and the above concentration inequalities
for $\bs{\lambda}_{C_q}^T \e$ and $\bs{\lambda}_{C_s}^T \e$
we have
\begin{align}
	 \left|\bs{b}_1^T \bs{e}\right| &
		\leq r_s \left|\bs{\lambda}_{C_q}^T \bs{e} - \frac{1}
		{2}\left(\mu_{qq} r_q - \tau \right)\right|
		+ r_q\left|\bs{\lambda}_{C_s}^T \bs{e}
		- \frac{1}{2}\left(\mu_{ss} r_s - \tau\right)\right|
		+ \left|\mu_{qs} r_s r_q - \bs{e}^T \bs{W}_{C_q,C_s}\bs{e}\right|
		\notag \\
		& \leq 3 r_s \max\bra{ \sqrt{\sigma_{qq}^2 \log n},
			\frac{\log n}{r_q} }
			+ 3 r_q \max\bra{ \sqrt{\sigma_{ss}^2 \log n},
			\frac{\log n}{r_s} } \notag \\
    & \hspace{0.95in}
			+  6 \max \bra{ \sqrt{\sigma_{qs}^2 r_q r_s \log n}, \log n}
		\notag \\
		& = O \rbra{ r_s\max\bra{ \sqrt{\tilde \sigma^2 \log n},
			\frac{\log n}{\hat r} }  },\label{eq:b1_e_orig}
\end{align}
w.h.p.
Finally, we bound $\|\bs{y}\|_{\infty}$  using \eqref{eq:b1_inf_orig} and \eqref{eq:b1_e_orig}:
$$
	\|\bs{y}^{q,s}\|_{\infty}
	\leq \frac{1}{r_s}\left( \|\bs{b}_1\|_{\infty}
		+ \frac{|\bs{b}_1^T \bs{e}|}{r_q +r_s}\right)
	= O \rbra{ \max \bra{  \sqrt{ \frac{ \tilde{\sigma}^2 \log n}{\hat{r}} },
		\frac{ \log n}{\hat{r}} } },
$$
w.h.p. Since this holds for any $q,s \in\{1,\dots,k\}$ such that $q \neq s$, we conclude that
\begin{equation*}
\|\bs{y}\|_{\infty} = O \rbra{ \max \bra{  \sqrt{ \frac{ \tilde{\sigma}^2 \log n}{\hat{r}} },
		\frac{ \log n}{\hat{r}} } },
\end{equation*}
w.h.p.
An identical argument shows that
\begin{equation*}
	\|\bs{z}^{q,s}\|_{\infty}
	= O \rbra{ \max \bra{  \sqrt{ \frac{ \tilde{\sigma}^2 \log n}{\hat{r}} },
		\frac{ \log n}{\hat{r}} } },
\end{equation*}
w.h.p.
We conclude that
\begin{equation*}
	\|\bs{y}^{q,s}\|_{\infty} +  \|\bs{z}^{q,s}\|_{\infty}
	= O \rbra{ \max \bra{  \sqrt{ \frac{ \tilde{\sigma}^2 \log n}{\hat{r}} },
		\frac{ \log n}{\hat{r}} } },
\end{equation*}
w.h.p.
\end{proof}

\section{Proof of Proposition~\ref{prop: lam}}
\label{app: lam}
We next prove Proposition~\ref{prop: lam}.
\begin{proof}
We follow the proof of Lemma 4.3 given by \cite{ames2014guaranteed}. Fix $q \in \{1,\dots,k\}$ and $i \in C_q$.
It follows from \eqref{eq: lam q} that
\[
\lambda_i = \sum_{j \in C_q} w_{ij} - \frac{1}{2r_q}\bs{e}^T\bs{W}_{C_q,C_q}\bs{e}-\frac{\tau}{2},
\]
for each \(i \in C_q\).
Applying \eqref{eq:hoeffding} with ${S} = \sum_{j \in C_q} w_{ij}$ and $T = n$ yields
\[
	\sum_{j \in C_q} w_{ij} \geq \mu_{qq} r_q
		- 6 \max \bra{ \sqrt{ \sigma_{qq}^2 r_q \log n }, \log n},
\]
w.h.p. Moreover,~by a similar argument, we have
\[
	\frac{1}{2r_q}\bs{e}^T\bs{W}_{C_q,C_q}\bs{e}
	\leq \frac{1}{2}\left(\mu_{qq} r_q +
	6 \max\bra{ \sqrt{\sigma_{qq}^2 \log n}, \frac{\log n}{r_q} }\right),
\]
w.h.p.
Combining the above inequalities shows that
\[
	\lambda_i \geq	r_q \rbra{ \frac{\mu_{qq}}{2} - \epsilon \gamma
		- O \rbra{ \max \bra{ \sqrt{  \frac{\sigma_{qq}^2 \log n }{r_q} },
			\frac{\log n }{r_q} } } },
\]
w.h.p.
Since $\gamma > 0$ by~\eqref{eq:smallest_largest}, this implies that
there exists constant $c>0$ such that if
\begin{equation} \label{eq: lam eps tmp}
	\epsilon \leq \frac{1}{2\gamma}
	\rbra{ \mu_{qq}  - c \max \bra{ \sqrt{  \frac{\sigma_{qq}^2 \log 		n }{r_q} },
			\frac{\log n }{r_q} } },
\end{equation}
then
$\lambda_i \geq 0$ w.h.p.
Applying the union bound over all \(q =1,2,\dots, k\) and
\(i\in C_q\) shows that each entry of $\bs{\lambda}_{C_q}$ is nonnegative w.h.p.~if \(\epsilon\) is chosen to satisfy \eqref{eq: lam eps tmp} for all $q$.
\end{proof}

\section{Proof of Proposition~\ref{prop:unique}}
\label{A_unique}

Our proof of Proposition~\ref{prop:unique} follows a similar structure to that of \citet[Lemma~4.4]{ames2014guaranteed}.

\begin{proof}
	Fix \eq{q \neq s} with \eq{q \in 1, \dots, k}.
	Applying \eqref{cluster sum} and \eqref{eq:hoeffding}
	with $S = \e^T \W_{C_q, C_s} \e$ and $T = n$,
	we have
	\begin{align*}
		r_s &\e^T \W_{C_q, C_q} \e
		- r_q \e^T \W_{C_q, C_s} \e \\
		& \ge (\mu_{qq} - \mu_{qs}) r_s r_q^2
		- 6 r_s \max \bra{ \sqrt{ \sigma_{qq}^2 r_q^2 \log n}, \log n}
		- 6 r_q \max \bra{ \sqrt{\sigma_{qs}^2 r_s r_q \log n}, \log n} \\
		&\ge r_q^2 r_s \rbra{ \gamma - 12 \max \bra{ \sqrt{\frac{\tilde\sigma^2 \log n}{\hat r^2}}, \frac{\log n}{\hat r^2}} },
	\end{align*}
	w.h.p.
	This implies that $r_s \e^T \W_{C_q, C_q} \e \ge r_q \e^T \W_{C_q, C_s} \e$ w.h.p.~if \eqref{eq: unique bound} is satisfied.
\end{proof}

\section{Proof of  Lemma~\ref{lem: S1}} 
\label{A_S1_bound}

We next prove Lemma~\ref{lem: S1}.
\begin{proof}
We will make repeated use of the following
lemma, which specializes the concentration inequality on the spectral norm
of a random symmetric matrix with i.i.d.~mean zero entries given by \citet[Corollary 3.12]{bandeira2016sharp}.

\begin{lemma} \label{lem: tropp tech}
	Let $\bs{A} = [a_{ij}] \in \Sigma^n$ be a random symmetric matrix with i.i.d.~mean zero entries $a_{ij}$ having variance at most
	$\sigma^2$ and satisfying $|a_{ij}| \leq 1$. Then there exists constant $C> 0$ such that
	\begin{equation} \label{eq:sparse_lemma}
	\pr \left\lbrace \|\bs{A}\| > C \max\left\lbrace \sqrt{ \sigma^2 n}, \sqrt{T} \right\rbrace \right\rbrace \leq nT^{-7}
	\end{equation}
	for all $T > 0$.
\end{lemma}

\begin{proof}
\newcommand{\ct}{\tilde c}
	{\bf (of Lemma~\ref{lem: tropp tech})}
	Corollary 3.12 of~\cite{bandeira2016sharp} establishes that for each $t>0$
	there exists $\tilde c = \tilde c(t) > 0$ such that
	\begin{equation} \label{eq:BVHtemp}
		\pr \bra{ \|\bs{A}\| \ge 3 \sqrt{\sigma^2 n} + t } \le n e^{-\ct t^2}.
	\end{equation}
	Here we have substituted the upper bound $\sigma^2 n \ge \tilde \sigma^2$, in place of $\tilde \sigma:= \max_i \sum_{j} \E[X_{ij}^2]$ in the original statement of Corollary~3.12.
	Let $t = (C-3) \max\{ \sqrt{\sigma^2 n}, \sqrt{\log T} \}$ where $C$ is chosen large enough that $\ct 	(C-3)^2 > 7$.
	In this case, \eqref{eq:BVHtemp} specializes to
	\[
		\pr \bra{\|\bs{A}\| \ge C \max\{ \sqrt{\sigma^2 n}, \sqrt{\log n}\} }
		\le n e^{-7 \max\{\sigma^2 n, \log n\}} \le n e^{-7 \log n} = n T^{-7}.
	\]
	This completes the proof.
\end{proof}

Before we continue with the derivation of the desired bound on \(\|\St_1\|\), we note that
the entries $[\St_1]_{ij}$ of $\St_1$ all satisfy $|[\St_1]_{ij}| \leq 1$ if we assume that \(w_{ij} \in [0,1]\)
for all \(i,j\); note that an identical argument establishes the result if we make the weaker assumption that
the entries of \(\W\) are bounded with high probability.
On the other hand, note that the
entries of $\St_1$ are not identically distributed (but are independent) since each $w_{ij}$ is sampled according to $\Omega_{qs}$, where $i \in C_q$, $j\in C_s$. However, we know that $\sigma^2_{qs} \le \tilde\sigma^2$ by our definition of $\tilde\sigma^2$. Moreover, $\E[{[\St_1]}_{ij}] = \E[\mu_{qs} - w_{ij}] = 0$.
Thus, we can apply Lemma~\ref{lem: tropp tech} to place a bound on $\|\St_1\|$.
Doing so establishes that~\eqref{eq: S1 bound} holds w.h.p.
\end{proof}

\section{Proof of  Lemma~\ref{lem: S2}} 
\label{A_S2_bound}

We conclude with the following proof of Lemma~\ref{lem: S2}.

\begin{proof}
Note that $\|\St_2 \| \le  \|\bs \lambda - \E[\bs \lambda] \| \sqrt{r_{k+1}}.$
Thus, it remains to bound $\|\bs \lambda - \E[\bs \lambda] \|$.
\newcommand{\bl}{\bs{\lambda}}
To do so, fix $q \in \{1,2,\dots, k\}$. Recall that
\[
	\bl_{C_q} - \E[\bl_{C_q}] = \frac{1}{r_q} ( \W_{C_q, C_q} \e - \mu_{qq} r_q \e )
		- \frac{1}{r_q^2} (\e^T \W_{C_q, C_q} \e - \mu_{qq} r_q^2) \e.
\]
Applying \eqref{eq:sparse_lemma} with $T = n$ establishes that
\begin{align*}
	\| \W_{C_q, C_q} \e - \mu_{qq} r_q \e \| &\le  \|\W_{C_q, C_q}  - \mu_{qq} \e\e^T \| \|\e\| \\
	& \le C \sqrt{r_q} \max\{ \sigma_{qq} \sqrt{ r_q}, \sqrt{\log n} \},
\end{align*}
w.h.p.
On the other hand, Bernstein's inequality establishes that
\[
	|  \e^T \W_{C_q, C_q} \e - \mu_{qq} r_q^2 | \le 6 \max \bra{ \sqrt{\sigma_{qq}^2 r_q \log n}, \log n },
\]
w.h.p.
Combining these two inequalities using the triangle inequality establishes that
\begin{align*}
	\| \bl_{C_q} - \E[\bl_{C_q}] \| \le &C \max \bra{ \sigma_{qq}, \sqrt{\frac{\log n}{r_q } } }
		+ 6 \max \bra { \sqrt{ \frac{\sigma_{qq}^2 \log n}{r_q^2}}, \frac{\log n}{r_q^{3/2} } } \\
		&= O \rbra{ \max \bra{ \sigma_{qq}, \sqrt{ \frac{ \log n}{\hat r} } } },
\end{align*}
w.h.p.
Finally, applying the union bound over all choices of $q$ shows that
\[
	\| \bl - \E [ \bl ] \|^2 = \sum_{q=1}^k \| \bl_{C_q} - \E[\bl_{C_q}] \|^2
		= O \rbra{  k \max \bra{ \sigma_{qq}^2, \frac{\log n}{\hat r} } },
\]
w.h.p.
This establishes that
\[
	\|\St_2\|^2 = O\rbra{k r_{k+1} \max \bra{ \sigma_{qq}^2, \frac{\log n}{\hat r} } },
\]
w.h.p., as required.
\end{proof}


\bibliography{SWKDC}

\end{document}